\documentclass[11pt, oneside, reqno]{amsart}
\usepackage{amsmath}
\usepackage{amsthm}
\usepackage{amssymb} 
\usepackage[mathscr]{euscript}
\usepackage{mathrsfs}
\usepackage{comment}
\excludecomment{toexclude} 
\usepackage{bm}
\usepackage{fancyhdr}
\usepackage{enumerate}
\usepackage{amscd}
\usepackage[all]{xy}
\usepackage{graphicx}
\usepackage{color}
\usepackage{hyperref}
\usepackage{stmaryrd} 
\hypersetup{
    colorlinks,
    citecolor=black,
    filecolor=black,
    linkcolor=black,
    urlcolor=black
}
\usepackage{enumerate}
\usepackage{tikz}
\usepackage[all]{xy}
\usepackage{graphicx}
\usetikzlibrary{backgrounds}
\usepackage{slashed}
\usepackage{enumitem,kantlipsum}
\usepackage{accents}
\usepackage{soul}

\renewcommand{\b}{\beta}

\newcommand{\D}{\Delta}

\newcommand{\Si}{\Sigma}

\newcommand{\be}{\begin{equation}}
\newcommand{\ee}{\end{equation}}
\newcommand{\bes}{\begin{equation*}}
\newcommand{\ees}{\end{equation*}}

\renewcommand{\to}{\rightarrow}

\newcommand{\Ga}{\mathsf{G}}

\newcommand{\Ric}{\mathrm{Ric}}

\theoremstyle{plain}
\newtheorem{lemma}{Lemma}[section]
\newtheorem{proposition}[lemma]{Proposition}
\newtheorem{theorem}[lemma]{Theorem}
\newtheorem{Theorem}{Theorem}

\newtheorem{corollary}[lemma]{Corollary}
\newtheorem{Corollary}[Theorem]{Corollary}

\newtheorem{conjecture}[Theorem]{Conjecture}
\newtheorem{question}[Theorem]{Question}

\numberwithin{equation}{section}

\newtheorem{manualtheoreminner}{Theorem}
\newenvironment{manualtheorem}[1]{%
  \renewcommand\themanualtheoreminner{#1}%
  \manualtheoreminner
}{\endmanualtheoreminner}

\theoremstyle{definition}
\newtheorem{remark}[lemma]{Remark}
\newtheorem{example}[lemma]{Example}
\newtheorem{definition}[lemma]{Definition}

\newtheorem{notation}[lemma]{Notation}
\newtheorem*{ack}{Acknowledgements}

\DeclareMathOperator{\pt}{\frac{\partial}{\partial t}}

\DeclareMathOperator{\ps}{\frac{\partial}{\partial s}} 
\DeclareMathOperator{\pst}{\frac{\partial^2}{\partial s\partial t}}

\DeclareMathOperator{\dt}{\frac{d}{d t}} 
\DeclareMathOperator{\ds}{\frac{d}{d s}}

\DeclareMathOperator{\tr}{tr} 
\DeclareMathOperator{\Div}{div}

\DeclareMathOperator{\Ker}{Ker}
\DeclareMathOperator{\Coker}{Coker}

\DeclareMathOperator{\Dim}{\mathrm{dim}}
 \DeclareMathOperator{\Span}{span}
\DeclareMathOperator{\Int}{Int}

\DeclareMathOperator{\C}{\mathcal{C}}

\newcommand{\too}{\longrightarrow}
\newcommand{\margin}[1]{\marginpar{\parbox{0in}{\parbox{0.8in}{\color{blue} \raggedright \scriptsize #1}}}}
\newcommand{\dvol}{\, dv}
\newcommand{\da}{\, d\sigma}

\usepackage[tmargin=1in,bmargin=1in,lmargin=1.5in,rmargin=1.5in]{geometry}
\begin{document}

\title[Local structure theory of Einstein manifolds with boundary]{Local structure theory of Einstein manifolds with boundary}

\author{Zhongshan An}
\address{Department of Mathematics, University of Michigan, Ann Arbor, MI, USA}
\email{zsan@umich.edu}
\author{Lan-Hsuan Huang}
\address{Department of Mathematics, University of Connecticut, Storrs, CT 06269, USA}
\email{lan-hsuan.huang@uconn.edu}
\thanks{The second author was partially supported by NSF DMS-2005588 and DMS-2304966.}

\begin{abstract}
We study local structure of the moduli space of compact Einstein metrics with respect to the boundary conformal  metric and  mean curvature. In dimension three, we confirm M. Anderson's conjecture  in a strong sense, showing that the map from Einstein metrics to such boundary data is generically a local diffeomorphism. In dimensions greater than three, we obtain similar results for Ricci flat metrics and negative Einstein metrics under new non-degenerate boundary conditions. 

 \end{abstract}

\maketitle

\tableofcontents
\addtocontents{toc}{\setcounter{tocdepth}{1}}  

\section{Introduction}
The existence of Einstein metrics and the structure of the moduli space have been fundamental questions in geometry and theoretical physics for decades (see, for example, the survey by S.-T.~Yau~\cite{Yau:1999}). While there has been significant progress in the existence theory for closed or complete manifolds without boundary, a full understanding of the moduli space of Einstein metrics remains elusive, except in dimensions 2, 3, and 4 (see surveys by M. Anderson~\cite{Anderson:2010} and C. LeBrun~\cite{LeBrun:1999}).

In 2008, M. Anderson proposed a research program to study Einstein metrics on compact manifolds \emph{with boundary}, recognizing that prescribing both the conformal class of the induced metric $g^\intercal$ and the mean curvature $H_g$ of the boundary $\Sigma$ in $(\Omega, g)$ is an elliptic boundary condition~\cite{Anderson:2008, Anderson:2008-banach}. We denote by $[\gamma]$ the \emph{pointwise} conformal class of metrics on $\Sigma$, where $\gamma_1 \cong \gamma_2$ if $\gamma_1 = u\gamma_2$ for some positive function $u$ on $\Sigma$. The pair $([g^\intercal], H_g)$ is referred to as the \emph{Anderson boundary data} of $(\Omega, g)$. Notably, P.~Gianniotis~\cite{Gianniotis:2016} established short-time existence and uniqueness of the Ricci flow with prescribed Anderson boundary data

While our paper exclusively  addresses Riemannian metrics, we note the growing interest in Anderson boundary data for Lorentzian metrics in general relativity and theoretical physics (see, e.g., \cite{An-Anderson:2021, Liu-Santos-Wiseman:2024, Witten:2021}). Due to the wave nature of the Einstein equation for Lorentzian metrics, there have been local-in-time existence results for the Cauchy or initial boundary value problem. However, the existence and structure theory for \emph{Riemannian} metrics remain relatively unexplored, which is the primary focus of our paper.

{

Another reason that finding Einstein metrics with prescribed Anderson boundary data is attractive is its potential connection to the existence question of conformally compact Einstein (CCE) metrics with prescribed \emph{conformal infinity}. This question has a long history, originating from theoretical physics, such as R. Penrose's work in the '60s, and has been reenergized by the AdS-CFT theory since the '90s. There have been some existence results (see, for example,~\cite{Fefferman:1976, Cheng-Yau:1980, LeBrun:1982, Fefferman-Graham:1985, Graham-Lee:1991, Anderson:2008, Gursky-Szekelyhidi:2020} and a recent survey \cite{Chang-Ge:2022}), but the problem is still wide open. While the existence problem of CCE metrics and our boundary value problem both concern the conformal structure of the ``boundary,'' these two problems are quite different; for example, the Einstein metric obtained in our setting is complete with boundary, whereas the metric in the other setting is geodesically complete. Nevertheless, our boundary value problem may be viewed as an analogy on a bounded region.}

To set the stage, we let $n\ge 3$ and $\Omega$ be a compact, connected $n$-dimensional manifold with a smooth nonempty boundary $\Sigma$. For $k\ge 2$ and $\alpha\in (0, 1)$, we denote the space of $\C^{k,\alpha}$ Riemannian metrics on $\Omega$ by $\mathcal M^{k,\alpha}(\Omega)$. We will work on the following subsets of $\mathcal M^{k,\alpha}(\Omega)$:
\begin{itemize}
\item The subset consisting of  Einstein  metrics with the fixed Einstein constant 
\[
\mathcal M^{k, \alpha}_\Lambda(\Omega) = \Big\{ g\in \mathcal M^{k,\alpha}(\Omega): \Ric_g = (n-1) \Lambda g\Big\}.
\]
\item The subset consisting of negative Einstein metrics 
\[
\mathcal M^{k, \alpha}_-(\Omega) = \Big\{ g\in \mathcal M^{k,\alpha}(\Omega): \Ric_g = \lambda g \mbox{ for some } \lambda<0\Big\}.
\]
\end{itemize}

Let $\mathscr D^{k+1,\alpha} (\Omega)$ denote the space of $\C^{k+1, \alpha}$ diffeomorphisms of $\Omega$ whose restriction on $\Sigma$ is the identity map of $\Sigma$.  For each fixed $k, \alpha$, we define the \emph{boundary  map} $\Pi:\mathcal{M}^{k,\alpha}(\Omega)/ \mathscr D^{k+1,\alpha}(\Omega)\to \mathcal S_1^{k,\alpha}(\Sigma)\times \C^{k-1,\alpha}(\Sigma) $ by 
\[
\Pi (g) = ([g^\intercal], H_g)
\]
where  $\mathcal S_1^{k,\alpha}(\Sigma)$ denotes the space of conformal classes $[\gamma]$ of $\C^{k,\alpha}$ Riemannian metrics on $\Sigma$. We identify $\mathcal S_1^{k,\alpha}(\Sigma)$ with the Banach space of Riemannian metrics $\gamma$ on $\Sigma$ whose determinant   $|\gamma|=1$ in the class $[\gamma]$. We will consider the boundary map restricted on either the moduli space $\mathcal{M}^{k,\alpha}_\Lambda(\Omega)/ \mathscr D^{k+1,\alpha}(\Omega) $ or $\mathcal{M}^{k,\alpha}_-(\Omega)/ \mathscr D^{k+1,\alpha}(\Omega)$.

Through out the paper, we assume  the relative fundamental group $\pi_1(\Omega, \Sigma)=0$, which means that $\Sigma$ is connected and the inclusion map $i:\Sigma\to \Omega$ induces a  surjection $i_*: \pi_1(\Sigma)\to \pi_1(\Omega)$. Anderson showed that $\mathcal{M}^{k,\alpha}_\Lambda(\Omega)/ \mathscr D^{k+1,\alpha}(\Omega) $ is an infinite dimensional smooth Banach manifold if $\pi_1(\Omega, \Sigma)=0$~\cite[Theorem 1.1]{Anderson:2008-banach}. (In contrast,  the Einstein moduli space of a \emph{closed} manifold is locally finite-dimensional~\cite{Berger-Ebin:1969, Koiso:1982}.) It is direct to verify that $\mathcal{M}^{k,\alpha}_-(\Omega)/ \mathscr D^{k+1,\alpha}(\Omega)$ is also an infinite dimensional smooth Banach manifold locally diffeomorphic to $ \big(\mathcal{M}^{k,\alpha}_{\Lambda}(\Omega)/ \mathscr D^{k+1,\alpha}(\Omega) \big) \times ( 0, \infty)$ (see Remark~\ref{re:harmonic}). On either moduli space, the boundary map $\Pi$ is a smooth map.

{Since an Einstein metric is of constant sectional curvature in dimension 3 and thus is more rigid,  Anderson made the following conjecture.}

\begin{conjecture}[Anderson {\cite[p. 2]{Anderson:2012}}]\label{conjecture}
Let $\Omega$ be $3$-dimensional manifold with smooth boundary $\Sigma$ satisfying $\pi_1(\Omega, \Sigma)=0$.  Then for each $\Lambda$ fixed, the boundary map $\Pi:\mathcal{M}^{k,\alpha}_\Lambda (\Omega)/ \mathscr D^{k+1,\alpha}(\Omega)\to \mathcal S_1^{k,\alpha}(\Sigma)\times \C^{k-1,\alpha}(\Sigma) $ defined by $\Pi (g) = ([g^\intercal], H_g)$ is regular\footnote{Recall  a map $f: M\to N$ between Banach manifolds $M, N$ is \emph{regular} at $x\in M$ (or $x$ is a \emph{regular point} of $f$) if the linearization $f'|_x$ is surjective with split kernel. If $x$ is not a regular point of $f$, $x$ is called a \emph{critical point}. } at generic metrics. 
\end{conjecture}
  As observed by Anderson, the ``genericity'' in the above conjecture cannot be removed~\cite{Anderson:2015}. Specifically, a round ball is a critical point of $\Pi$ because a result of B.~Ammann, E. Humbert, and M. Ould Ahmedou~~\cite{Ammann-Humbert-Ahmedou:2007} says that any conformal immersion $F$ from a round sphere $(S^{n-1}, g_{S^{n-1}})$   into the Euclidean space $(\mathbb R^n, \bar g)$ must satisfy
\begin{align}\label{eq:AHO}
	\int_{S^{n-1}} Z(H\circ F)\da_{F^* \bar g} =0
\end{align}
where  $Z$ is   any conformal Killing vector field  on the round sphere. Therefore, $Z(H\circ F)$ must change signs. Since some small perturbations of $H_{S^{n-1}}  =n-1$ does not satisfy \eqref{eq:AHO}, it implies that  the boundary map cannot be locally surjective.

We confirm Conjecture~\ref{conjecture} in a stronger form, showing that $\Pi$ is generically a local diffeomorphism. Consequently,  an open neighborhood of a generic metric in the moduli space of 3-dimensional Einstein metrics is uniquely parametrized by the Anderson boundary data. Our result  implies that the image space $\Pi(\mathcal{M}^{k,\alpha}_\Lambda(\Omega)/ \mathscr D^{k+1,\alpha}(\Omega))$, which is already known to be a variety of finite codimension, in fact contains a dense subset that is open in  the codomain $\mathcal S_1^{k,\alpha}(\Sigma)\times \C^{k-1,\alpha}(\Sigma)$. It also provides some answers to the question raised in~\cite[pp. 2011-2012]{Anderson:2008-banach}  about the structure of the image space. \begin{Theorem}\label{th:generic0}
 Let  $\Omega$ be $3$-dimensional with smooth boundary $\Sigma$ satisfying $\pi_1(\Omega, \Sigma)=0$. Then for each fixed $\Lambda$, the  boundary map $\Pi$  is a local diffeomorphism on  an open dense subset of $\mathcal M^{k,\alpha}_\Lambda(\Omega)/\mathscr D^{k+1,\alpha}(\Omega)$. 
  \end{Theorem}
  
 Our approach also extends to higher dimensions $n\ge 4$ under non-degenerate boundary conditions that we will introduce below. For $\bar g\in \mathcal M_\Lambda^{k,\alpha} (\Omega)$, we define an elliptic operator on the boundary $\Sigma$ as follows: for a scalar function $v$, 
  \begin{align}\label{eq:bdry-op}
  L_\Sigma v:= -\Delta_\Sigma v - \tfrac{1}{n-2} R_\Sigma v
  \end{align}
   where $\Delta_\Sigma, R_\Sigma$ are respectively the Beltrami-Laplace operator and the scalar curvature of the induced metric $\bar g^\intercal$ on the boundary.\footnote{The operator $L_\Sigma$ comes from the linearization of the scalar curvature $R_\Sigma$ among conformal deformations; more precisely, $L_\Sigma v  = \frac{1}{n-2}(R_\Sigma)'|_{\bar g^\intercal} (v\bar g^\intercal)$.} 
\begin{definition}
\begin{enumerate}
\item  We say $(\Omega, \bar g)$ has  \emph{non-degenerate boundary} if       
   \begin{align}	
  	&\int_\Sigma vL_\Sigma v\da_{\bar g} >0 \quad \mbox{ for all $v\not \equiv 0$ satisfying $\int_\Sigma v H_{\bar g}  \da_{\bar g} = 0$} \tag{$\star$} \label{eq:bd1}
  \end{align}
  where $H_{\bar g}$ is the mean curvature of $\Sigma$ in $(\Omega, \bar  g)$.   
\item    We say $(\Omega, \bar g)$ has  \emph{$H$-non-degenerate boundary} if     
   \begin{align}	
  	&\int_\Sigma  H_{\bar g} (vL_\Sigma v)\da_{\bar g} >0 \quad \mbox{ for all $v\not \equiv 0$ satisfying $\int_\Sigma v H_{\bar g}  \da_{\bar g} = 0$} \tag{$\star_H$}. \label{eq:bd2}
  \end{align}
\end{enumerate}
   \end{definition}
 
{We note that if  $R_\Sigma, H_{\bar g}$ are nonzero constants, then both non-degenerate boundary conditions become the condition that the \emph{second eigenvalue} of $L_\Sigma$ is positive (while the first eigenvalue is $-\tfrac{1}{n-2} R_\Sigma$ with the first eigenfunction being $1$). } For instance, if $\Sigma$ is the round sphere, $L_\Sigma$ becomes the spherical Laplace operator whose kernel comprising spherical harmonics of degree 1. Thus, the round ball fails to satisfy the non-degenerate condition.

We first consider $\Lambda=0$; namely, the space of the Ricci flat metrics $\mathcal M^{k,\alpha}_0 (\Omega)$. Define 
\begin{align*}
&\widehat {\mathcal M}^{k,\alpha}_0 (\Omega) = \Big\{ g\in \mathcal M^{k,\alpha}_0 (\Omega):\, \mbox{either $(\Omega, \bar g)$ has non-degenerate boundary} \\
&\quad  \mbox{ or $(\Omega, \bar g)$ has $H$-non-degenerate boundary and strictly mean convex boundary}\Big\}.
\end{align*}
 It is clear that $\widehat {\mathcal M}^{k,\alpha}_0 (\Omega)$ is an open subspace of $\mathcal M^{k,\alpha}_0(\Omega)$.

  \begin{Theorem}\label{th:generic-higher}
 Let $n \ge  4$  and $\Omega$ be $n$-dimensional with smooth boundary $\Sigma$ satisfying $\pi_1(\Omega, \Sigma)=0$.   Then the  boundary map $\Pi: \widehat {\mathcal M}^{k,\alpha}_0 (\Omega)/ \mathscr D^{k+1,\alpha}(\Omega)\to \mathcal S_1^{k,\alpha}(\Sigma)\times \C^{k-1,\alpha}(\Sigma)$ is a local diffeomorphism on  an open dense subset. 

 \end{Theorem}

Recall the space of negative Einstein metrics $\mathcal M_{-}^{k,\alpha}(\Omega)$. Let $\widehat {\mathcal M}^{k,\alpha}_- (\Omega)$  be the open subset of $\mathcal M^{k,\alpha}_- (\Omega)$ that consists of metrics $\bar g$ so that $(\Omega, \bar g)$ has $H$-non-degenerate boundary and strictly mean convex boundary.  
\begin{Theorem}\label{th:general}
Let $n\ge 4$ and $\Omega$ be $n$-dimensional with smooth boundary $\Sigma$ satisfying $\pi_1(\Omega, \Sigma)=0$. Then the boundary map $\Pi :\widehat {\mathcal M}_{-}^{k,\alpha}(\Omega) /\mathscr D^{k+1,\alpha}(\Omega)\to \mathcal S_1^{k,\alpha}(\Sigma)\times \C^{k-1,\alpha}(\Sigma)$ is regular on an open dense subset $\mathcal W$. 

More specifically, for  $\llbracket g \rrbracket\in \mathcal W$ and  $(\gamma, \phi)$ sufficiently close to $([g^\intercal], H_g)$ in $\mathcal S_1^{k,\alpha}(\Sigma)\times \C^{k-1,\alpha}(\Sigma)$, there exists a unique path $g_s$, for $|s|$ small, in a neighborhood of $g$ in $\widehat{ \mathcal M}_{-}^{k,\alpha}(\Omega) /\mathscr D^{k+1,\alpha}(\Omega)$ so that $g_s$ realize the same Anderson boundary data  $(\gamma, \phi)$, satisfy
\[
	\Ric_{g_s} = (1+s)\Lambda \left( \frac{V_g}{V_{g_s}}\right)^{\frac{2}{n}} g_s,
\]
and $g_s$ is not isometric to $g_{s'}$ for all $s\neq s'$, where $\Lambda$ is the constant such that $\Ric_g = (n-1)\Lambda g$. 
\end{Theorem}

 Note that the Einstein constant $\Lambda_s:=(1+s)\Lambda \left( \frac{V_g}{V_{g_s}}\right)^{\frac{2}{n}}$ of $g_s$ in Theorem~\ref{th:general} depends on the volume $V_{g_s}$, which we have no control over.  Potentially, $\Lambda_s$ can be the same for different values of $s$, and thus we could not improve the boundary map $\Pi$ to a local \emph{diffeomorphism} when restricted to a smaller space $\widehat {\mathcal M}_{\Lambda}^{k,\alpha}(\Omega) /\mathscr D^{k+1,\alpha}(\Omega)$ with some fixed Einstein constant $\Lambda$. This highlights the subtlety for negative Einstein metrics when compared with the Ricci flat case in Theorem~\ref{th:generic-higher}. The assumption that $\Lambda<0$ comes in mysteriously in the application of the Green-type identity in \eqref{eq:bulk-sign}. It remains an open question whether an analogous statement holds for $\Lambda>0$.

 Given that 3-dimensional Einstein metrics have constant sectional curvature, the problem of finding an Einstein metric with prescribed Anderson boundary data in dimension 3 is also related to  the question of conformal immersion  into spaceforms. Specifically,  Anderson raised the following question about immersions into Euclidean space~\cite[p. 143]{Anderson:2015}.
 
\begin{question}\label{qe}
Given a closed orientable surface $\Sigma$ and arbitrary Anderson boundary data $([\gamma], \phi)$ on $\Sigma$, does there exist an immersion $f: \Sigma\to \mathbb R^3$ with  $([f^*\big(g_0^\intercal \big) ], f^* \big(H_{f}\big) ) = ([\gamma], \phi)$, where $g_0^\intercal$ and $H_{f}$ are respectively the induced metric and mean curvature of the immersed image $f(\Sigma)$.
\end{question}

Anderson~\cite{Anderson:2015} solved a related problem where prescribing mean curvature is weakened to prescribing mean curvature up to certain affine functions and also illustrated the contrasting challenges for Question~\ref{qe} (, essentially, the map $\Phi$ as defined below is of zero degree, and thus it is inadequate to derive existence via a degree argument.)  Below, we give local existence of Question~\ref{qe} for Alexandrov immersions into a general spaceform. Recall that  $f: \Sigma\to S_\Lambda$ is an \emph{Alexandrov immersion} if there is a compact $3$-manifold $\Omega$ with boundary $\Sigma$ and an immersion $F:\Omega\to S_\Lambda$ such that $F|_\Sigma = f$.  

In the next two results, we let $\Omega$ be a simply-connected 3-dimensional manifold whose boundary $\Sigma$ is a  smooth connected surface. Suppose $F_0$ is a smooth immersion of $\Omega$ into  a 3-dimensional spaceform $S_\Lambda$.  Fix $p\in \Omega$. Denote the space of  $\C^{k+1,\alpha}$ immersions that fix the $1$-jet at $p$ by 
\begin{align*}
\mathrm{Imm}^{k+1,\alpha}_p(\Omega, S_\Lambda)  =\big \{ F:\Omega \to S_\Lambda :  \mbox{$F\in \C^{k+1,\alpha}$ with } F(p) = F_0(p), dF|_p = dF_0|_p \big\}.
\end{align*}
Denote the space of $\C^{k+1,\alpha}$ diffeomorphisms of $\Omega$ that fix the boundary and the $1$-jet at $p$ by
 \[
 \mathscr D_p^{k+1,\alpha} (\Omega) = \big\{ \psi \in \mathscr D^{k+1,\alpha}(\Omega): \psi(p)=p, \,  d\psi|_p= \mathrm{Id} \big\}.
 \]
We show in Lemma~\ref{le:embedding} that $\mathrm{Imm}^{k+1,\alpha}_p(\Omega, S_\Lambda)/\mathscr D_p^{k+1,\alpha} (\Omega) $ is locally diffeomorphic to $\mathcal M^{k,\alpha}_\Lambda/\mathscr D^{k+1,\alpha}(\Omega)$, which gives the following immediate consequence from Theorem~\ref{th:generic0}.

\begin{Corollary}\label{co:bdrymap0}
Let $\Omega$ be a simply-connected $3$-dimensional manifold immersed in $S_{\Lambda}$,  with smooth connected boundary $\Sigma$. For each fixed $\Lambda$, define the map 
\begin{align*}
&\Phi: \mathrm{Imm}^{k+1,\alpha}_p(\Omega, S_\Lambda) /\mathscr D_p^{k+1,\alpha}(\Omega) \to\mathcal S_1^{k,\alpha}(\Sigma)\times \C^{k-1,\alpha}(\Sigma)\\
&\quad \Phi (F)= \big([F^*\big(g_\Lambda^\intercal \big) ], F^* \big(H_{F}\big) \big)
\end{align*}
where $g_\Lambda^\intercal$ and $H_{F}$ are respectively the induced metric and mean curvature of the immersed image $F(\Sigma)\subset (S_\Lambda, g_\Lambda)$. Then $\Phi$ is a local diffeomorphism on an open dense subset. 
\end{Corollary}


We can further strengthen Corollary~\ref{co:bdrymap0} by removing genericity for star-shaped regions. Let $\widehat {\mathrm{Emb}}^{k+1,\alpha}_p(\Omega, \mathbb R^3) \subset \mathrm{Imm}_p^{k+1,\alpha}(\Omega, \mathbb R^3)$  be the subspace consisting of embeddings $F$ so that $F(\Omega)$ is star-shaped  and is not a round ball.

\begin{Theorem}\label{th:bdrymap}
Let $\Omega$ be a simply-connected $3$-dimensional submanifold of $\mathbb R^3$ with smooth connected boundary $\Sigma$.  The map $\Phi$ restricted on $\widehat{\mathrm{Emb}}_p^{k+1,\alpha}(\Omega, \mathbb R^3) /\mathscr D_p^{k+1,\alpha}(\Omega)$ is a local diffeomorphism. 
\end{Theorem}

Regarding Question~\ref{qe}, our results indicate that an arbitrary Anderson boundary data $([\gamma],\phi)$ sufficiently close to that of a star-shaped embedding $F$, which is not a round sphere, is realized by a geometrically unique star-shaped embedding in a neighborhood of $F$.   

Although we only focus on studying regular points of the boundary map in this paper,  we will expand the current approach to analyze critical points and in particular construct Ricci flat metrics perturbed from round balls with prescribed boundary mean curvature in an upcoming paper.

\subsection{The motivating example and organization of the paper}

We first discuss central ideas of the approach for the three-dimensional result (Theorem~\ref{th:generic0}) and the case of Ricci flat (Theorem~\ref{th:generic-higher}). Then we discuss the necessary modification for the case of negative Einstein constant (Theorem~\ref{th:general}) at  the end of the section. The approach is inspired by our recent work on the Bartnik static extension conjecture in general relativity~\cite{An-Huang:2022, An-Huang:2022-JMP, An-Huang:2024}.

 Let $\Omega$ be a compact manifold with boundary $\Sigma$. We can formulate the problem as follows: Given a Riemannian metric $\gamma$  with determinant $1$ and a scalar function $\phi$ on $\Sigma$, we would like to solve the geometric boundary value problem $T(g) = (0, \gamma, \phi)$, where the operator $T$, defined between suitable Banach spaces, is given by
\begin{align*}
T (g)= \left\{ \begin{array}{ll} \Ric_g - (n-1)\Lambda  g  \quad \mbox{ in } \Omega \\
	\left\{ \begin{array}{l} |g^{\intercal}|^{-\frac{1}{n-1}} g^\intercal\\ H_g 
	\end{array}\right. \quad \mbox{ on } \Sigma
	\end{array} \right..
\end{align*}
 The operator modified by the harmonic gauge 
\begin{align*}
	&T^\Ga (g ) = \left\{ \begin{array}{ll} \Ric_{g}  - (n-1)\Lambda \eta(g) g + \mathcal D_g  \beta_{\bar g} g\quad \mbox{ in } \Omega \\
	\left\{ \begin{array}{l} \beta_{\bar g} g \\   |g^{\intercal}|^{-\frac{1}{n-1}} g^\intercal\\ H_g 
	\end{array}\right. \quad \mbox{ on } \Sigma
	\end{array} \right.
\end{align*}
is elliptic and has a Fredholm index of $0$ \cite{Anderson:2008, Anderson:2008-banach} (also see Section~\ref{se:gauge} for details). The goal is to show that its linearization at $\bar g$ has a trivial kernel for a generic $\bar g\in \mathcal M^{k,\alpha}_\Lambda(\Omega)$. Then one applies the Inverse Function Theorem.

To motivate the arguments in Section~\ref{se:generic}, which establishes the triviality of the kernel, we consider the following simple elliptic boundary value problem. Let $\bar g$ be a Riemannian metric and $u$ a scalar or vector-valued function on $\Omega$. We define
\begin{align}\label{eq:example0}
	B u = \left\{ \begin{array}{ll} \Delta_{\bar g} u + f u  & \mbox{ in } \Omega\\
	u & \mbox{ on } \Sigma \end{array}\right.,
\end{align}
where $f$ is a smooth scalar function on $\Omega$. It is well-known that the kernel of $B$ is generically trivial for inward domain deformations. Below, we present a proof that is different from standard proofs, reflecting our method for the Ricci boundary value problems within this simple setting.

Consider a smooth one-parameter family of inward deformations of $\Omega$ by $\psi_t: \Omega \to \Omega_t$ for $t\in [0, \delta)$, such that $\psi_0=\mathrm{Id}$, $\left. \frac{\partial}{\partial t}\right|_{t=0} \psi_t = V$, and $V|_\Sigma = -\nu$, where $\nu$ is the outward unit normal to $\Sigma$. Denote the pull-back metrics  on $\Omega$ by $g(t) = \psi_t^* (\bar g|_{\Omega_t})$. Studying the kernel of $B$ on $\Omega_t$ is equivalent to studying the kernel of the pull-back operators $B_t$ on the fixed $\Omega$:
\begin{align*}
	B_t u = \left\{ \begin{array}{ll} \Delta_{g(t)} u + (\psi_t^* f) u  & \mbox{ in } \Omega\\
	u & \mbox{ on } \Sigma. \end{array}\right.
\end{align*}

\begin{lemma}\label{le:motivation}
There is an open dense subset $J_0\subset [0, \delta)$ such that $\Ker B_t=\{0\}$ for $t\in J_0$.  
\end{lemma}
\begin{proof}
Define the nullity function $N(t) = \Dim \Ker B_t$. One can show that the set $J_0=\{ t\in [0, \delta): N(t) \mbox{ is constant in an open neighborhood of $t$}\}$ is an open and dense subset. (See Lemma~\ref{le:nullity}.) For $a\in J_0$, there is a decreasing sequence $t_j\searrow a$ such that $N(t_j)= N(a)$.  Without loss of generality, we assume $a=0$. 

Suppose $N(0) >0$. For each $t_j$, let $u(t_j)$ be a nontrivial kernel element of $B_{t_j}$ with the normalization $\| u(t_j)\|_{\C^0(\Omega)}= 1$. By elliptic estimate, there is a convergent subsequence, labeled still as $u(t_j)$, that converges to $u \neq 0$ with $B u=0$ as $t_j\searrow 0$.

Because $\Omega_{t_j}\subset \Omega$, we can restrict $u$ on $\Omega_{t_j}$ and pull it back to $\Omega$, denoted by $\psi_{t_j}^* u $, which satisfies
\[
\Delta_{g(t_j)} \psi_{t_j}^* u + (\psi_{t_j}^* f) (\psi_{t_j}^* u)= 0 \quad \text{in } \Omega \quad \text{for all } j.
\]
Subtracting the equation for $u(t_j)$, we get 
\[
\Delta_{g(t_j)} (u(t_j)- \psi_{t_j}^* u )+ (\psi_{t_j}^* f) (u(t_j)-\psi_{t_j}^* u)= 0 \quad \text{in } \Omega.
\] 
Taking the difference quotient in $t$ and evaluating at $t=0$, we obtain 
\[
\Delta_{\bar g} (u'(0) - V(u)) + f (u'(0)- V(u)) = 0
\]
where $u'(0) := \lim_{j\to \infty} \frac{u(t_j) - u}{t_j}$. Note $u'(0)$ exists. Cf. Proposition~\ref{pr:kernel} in a more general setting.

To complete the argument, we apply the Green identity for $u$ and $w:= u'(0)-V(u)$, obtaining
\begin{align}\label{eq:motivation}
0 = \int_\Omega \Big ( w\Delta u - u \Delta w\Big) \dvol= \int_\Sigma\Big( w \nu(u) - u \nu (w)  \Big)\da =   \int_\Sigma (\nu(u))^2\da
\end{align}
where we use that $u$ has zero Dirichlet boundary value and $w|_\Sigma = \nu (u)$. This shows that $u$ satisfies the Cauchy boundary condition $u=0, \nu(u)=0$ on $\Sigma$, and thus $u\equiv 0$ by unique continuation, contradicting our initial assumption.
\end{proof}

We extend the  approach in the previous simple situation to our geometric boundary value problem. In addition to the subtlety of gauge choices for our geometric tensor equation, the general argument is significantly more challenging due to several factors and requires novel ingredients, which will be highlighted below.

In Section~\ref{se:pre}, we consider the \emph{Green-type identity} for the Einstein operator with respect to the Anderson boundary data, which will play the role as the second equality in \eqref{eq:motivation}. Indeed, the Ricci operator (as it appears in the operator $T$ defined above)  is not self-adjoint with respect to the Anderson boundary data. By the second variation of the Einstein-Hilbert functional, we use instead the (linearized) Einstein operator $P(h)$  in the Green-type identity, which  is self-adjoint with respect to the Anderson boundary data and  is also ``algebraically equivalent'' to the linearized Ricci operator. See Proposition~\ref{pr:Green-conf}.

Next, we use Green-type identity to find the correct ``Cauchy boundary condition'' in this setting, which we refer to as the (linearized) \emph{conformal Cauchy boundary condition}.  The argument of applying the Green-type identity to obtain the conformal Cauchy boundary condition as in \eqref{eq:motivation} yields a stronger result in the three-dimensional case than in dimensions greater than three, where, in the latter case, the non-degenerate boundary condition comes into play. In Section~\ref{se:conf}, we discuss unique continuation, referred to as \emph{infinitesimal rigidity}, with respect to the conformal Cauchy boundary condition. While we find that infinitesimal rigidity holds in most cases, we also observe that it fails for round balls, explaining the nature of the boundary map fails to be regular at the round ball. We then prove Theorems~\ref{th:generic0}, \ref{th:generic-higher}, and \ref{th:bdrymap} in Section~\ref{se:generic}.

In Section~\ref{se:nonzero}, we prove Theorem~\ref{th:general}; specifically, the case where $\Lambda <0$ and $n\ge 4$. A fundamental difference in the case of $\Lambda \neq 0$ compared to $\Lambda=0$ is that  the Hilbert-Einstein functional is not invariant under metric dilation. Consequently, a ``zero mean value'' property (Lemma~\ref{le:trH}, cf. Lemma~\ref{le:vanish}), necessary for incorporating the non-degenerate boundary conditions, does not hold when $\Lambda\neq 0$. Instead, the mean value is linked to volume deformation. This observation motivates the consideration of the Ricci equation with varying Einstein constants in conjunction with volume, modified  from the operator $T$. Specifically, the new map is defined for a pair consisting of a Riemannian metric $g$ and a real number $\eta$ as follows:
 \begin{align*}
	\overline T(g, \eta) =  \left\{\begin{array}{ll} \Ric_g - (n-1) \Lambda \eta g \quad \mbox{ in } \Omega \\
	\left\{ \begin{array}{ll} |g|^{-\frac{1}{n-1}}g^\intercal \\
	H_g \end{array} \right. \mbox{ on } \Sigma\\
		2\eta^{\frac{n}{2}} V_g 
	\end{array}  \right..
\end{align*}
We then employ the domain deformation approach to show that the gauged version of $\overline{T}$ has a trivial kernel generically and obtain Theorem~\ref{th:general} via Inverse Function Theorem. 
 
In Section~\ref{se:example}, we present the family of $n$-dimensional Euclidean Schwarzschild metrics for $n\ge 4$ as examples for the non-degenerate boundary condition. Such solutions are  called gravitational instantons in general relativity. We demonstrate that for $n=4$, all sublevel sets of the radius function have non-degenerate boundary. We extend our discussion to higher dimensions, as well as the analogous cases of $\Lambda\neq 0$, namely the Euclidean anti-de Sitter Schwarzschild and the Euclidean de~Sitter Schwarzschild, to identify necessary and sufficient conditions for the radius to have non-degenerate boundary.

 \begin{ack}
We thank Michael Anderson for suggesting those examples in Section~\ref{se:example} and for providing valuable comments on a previous version.
\end{ack}

\section{Green-type identity and harmonic gauge}\label{se:pre}

Throughout the paper, let $n\ge 3$ and $\Omega$ be an $n$-dimensional compact manifold with a nonempty smooth  boundary $\Sigma$. Assume that $\Omega$ carries a smooth Riemannian metric $\bar g$ with $\Ric_{\bar g} = (n-1) \Lambda \bar g$, referred to as a \emph{background metric}. {Fixing an atlas of $\Omega$ such that the components of $\bar g$ are analytic (e.g. using harmonic coordinates by \cite[Theorem 5.2]{DeTurck-Kazdan:1981}), we use  $\C^{k,\alpha}$ to denote the H\"older spaces (for both functions and tensors) with respect to that atlas.} We denote $\bar g^\intercal$ as the induced metric on $\Sigma$ and $\nu$ as the \emph{outward} unit normal of $\bar g$ to $\Sigma$. The second fundamental form is defined as $A= \nabla_\Sigma \nu$, and the mean curvature is denoted by $H = \tr_\Sigma A$. In the subsequent sections, additional assumptions on $\Omega$ will be imposed as needed; see Notation~\ref{notation} and Notation~\ref{no:foliation}.

\subsection{The setup of the geometric boundary value problem}
 Fixing $k\ge 2$ and $\alpha \in (0, 1)$, we denote by $\mathcal M^{k,\alpha}(\Omega)$ the space of $\C^{k,\alpha}$ Riemannian metrics on $\Omega$. For $\Lambda\in\mathbb R$, let  $\mathcal  M^{k,\alpha}_\Lambda(\Omega)\subset \mathcal M^{k,\alpha}(\Omega)$ be the subspace of  metrics~$g$ satisfying $\Ric_{g} = (n-1)\Lambda g$ in  $\Omega$. Fixing $\Omega$, we always assume  $\mathcal  M^{k,\alpha}_\Lambda(\Omega)\neq \emptyset$. 
 
 Given a Riemannian metric $\gamma$ and a scalar function $\phi$ on $\Sigma$, we would like to solve for $g\in \mathcal M^{k,\alpha}_\Lambda(\Omega)$ that satisfies the boundary condition
\begin{align}\label{eq:bdry}
		( [g^\intercal], H_g ) &= ([\gamma], \phi)  \quad \mbox{ on } \Sigma
\end{align}
where $[\gamma]$ denotes the conformal class such that $\gamma_1\cong \gamma_2$  if and only if $\gamma_1 = u\gamma_2$ for some positive function $u$ on $\Sigma$. Let $\mathcal S_1^{k,\alpha}(\Sigma)$ denote the space of conformal classes of $\C^{k,\alpha}$ Riemannian metrics on $\Sigma$. {Define the determinant of $\gamma$ with respect to the background metric $\bar g^\intercal$ by  $|\gamma|=\left(\frac{\da_{\gamma}}{\da_{\bar g^\intercal}}\right)^2$.} (In particular, $|\bar g^\intercal|=1$.)   We can identify $\mathcal S_1^{k,\alpha}(\Sigma)$ with the space of $\C^{k,\alpha}$ Riemannian metrics $\gamma$ on $\Sigma$ whose determinant is identically $1$. Note $\mathcal S_1^{k,\alpha}(\Sigma)$ is a Banach manifold, whose tangent space at  $\bar g^\intercal$, denoted by $T \mathcal S_1^{k,\alpha}(\Sigma)$, consists of symmetric $(0,2)$-tensors that are traceless with respect to $\bar g^\intercal$. See  Lemma~\ref{le:Banach} below. 

Finding an Einstein metric $g$ with the boundary condition \eqref{eq:bdry} can be reformulated into solving $T(g) = (0, \gamma, \phi)$ where the operator $T:\mathcal M^{k,\alpha} (\Omega) \to \C^{k-2,\alpha}(\Omega)\times \mathcal S_1^{k,\alpha}(\Sigma)\times \C^{k-1,\alpha}(\Sigma)$ is defined by
\begin{align}\label{eq:T}
T(g ) = \left\{ \begin{array}{ll} \Ric_g - (n-1)\Lambda  g \quad \mbox{ in } \Omega \\
	\left\{ \begin{array}{l} |g^{\intercal}|^{-\frac{1}{n-1}} g^\intercal\\ H_g 
	\end{array}\right. \quad \mbox{ on } \Sigma
	\end{array} \right..
\end{align}
(Technically, we should denote the operator $T$ as $T_{\Lambda}$, but we omit the subscript $\Lambda$ for clarity, as we consistently fix $\Lambda$, making it clear from the context.)

The linearized map at  $\bar g$ is denoted by  $L: \C^{k,\alpha}(\Omega)\to\C^{k-2,\alpha}(\Omega)\times T \mathcal  S_1^{k,\alpha}  (\Sigma)\times \C^{k-1,\alpha}(\Sigma)$ 
\begin{align}\label{eq:L}
	L (h ) = \left\{ \begin{array}{ll} \Ric'(h) - (n-1)\Lambda h  \quad \mbox{ in } \Omega \\
	\left\{ \begin{array}{l}   h^\intercal - \tfrac{1}{n-1} (\tr h^\intercal) \bar g^{\intercal} \\ H'(h) 
	\end{array}\right. \quad \mbox{ on } \Sigma
	\end{array} \right..
\end{align}

Here and below, we omit the subscript $\bar g$ when linearizing at $\bar g$. Please refer to Section~\ref{se:va} for the formulas of the linearized operators, including $\Ric'(h)$ and $H'(h)$.

The main emphasis of this paper is to analyze of the kernel of $L$, denoted by $\Ker L$. To provide context for our subsequent results on $\Ker L$, we present two prototype examples of kernel elements. Throughout the paper, we denote the \emph{restriction} of a vector field $X$ along the boundary as $X|_\Sigma = f\nu + X^\intercal$, where $X^\intercal$ is the tangential component  to $\Sigma$.

\begin{example}[Trivial kernel] 
Since $T$ is geometric, if $T(g) = (0, \gamma, \phi)$, then $T(\psi^* g ) = (0, \gamma, \phi)$ for any diffeomorphism $\psi: \Omega\to \Omega$ that fixes the boundary. Consider the diffeomorphisms $\psi_t$ generated by a vector field $X$ with $X|_\Sigma = 0$. Differentiating $T(\psi_t^* g)$ gives $L (\mathscr L_X \bar g)=0$. Thus, we consider $h\in \Ker L$ as \emph{trivial} if
$h = \mathscr L_X \bar g$ for some $X$ satisfying $X|_\Sigma =0$.

Later, in Corollary~\ref{co:rigid}, we show that $\Ker L$ is generically trivial in this sense.
\qed
 \end{example}

 Recall that a vector field $Z$ on an $(n-1)$-dimensional Riemannian manifold $(\Sigma, \gamma)$ is called a \emph{conformal Killing vector field} if $ \mathscr L_Z  \gamma =\left( \tfrac{2}{n-1}\Div_\gamma Z\right) \gamma$. The space of conformal Killing vector fields on the unit round sphere $(S^{n-1}, g_{S^{n-1}})$ consists of ${ Z_1, \dots, Z_n}$ and rotation vector fields, where $Z_i =\left(\tfrac{\partial}{ \partial x_i}\right)^\intercal$ is the tangential component of the $\mathbb R^n$ coordinate vector on the unit sphere centered at the origin. 

\begin{example}[Nontrivial kernel elements for round spheres]\label{ex:conformal}
Let $(\Omega, \bar g)$ be a compact manifold with boundary $\Sigma$. Suppose $(\Sigma, \bar g^\intercal)$ is the round sphere $(S^{n-1}, g_{S^{n-1}})$ and has constant mean curvature in $\Omega$. Let $X$ be any vector field in $\Omega$ such that $X|_\Sigma = Z\in \Span \{Z_1, \dots, Z_n\}$. Denote by $h = \mathscr L_X \bar g$. We compute
 \begin{align*}
h^\intercal &= \mathscr L_Z g_{S^{n-1}} = \left( \tfrac{2}{n-1}\Div_{S^{n-1}} Z\right) g_{S^{n-1}} = \tfrac{1}{n-1} (\tr h^\intercal ) \bar g^\intercal, \\
H'(h) &= Z(H_{\bar g}) = 0.
\end{align*}
In other words, $h$ is a \emph{nontrivial} kernel element of $L$.

We also note that if, in addition, $\Sigma$ is umbilic, then we compute
\begin{align}\label{eq:2ffK}
	A'(h)_{\alpha\beta} =Z^\delta A_{\alpha \beta| \delta} + Z_{\gamma|\beta} A_{\alpha \gamma} + Z_{\gamma|\alpha} A_{\beta \gamma} = \tfrac{1}{n-1} (\tr h^\intercal) A_{\alpha \beta}
\end{align}
where the subscript $_{|}$ denotes the covariant derivative of $\Sigma$. This shows that such $h$ also satisfies the \emph{conformal Cauchy boundary} condition, which is particularly relevant to Theorem~\ref{th:infinitesimal}.

We will prove that those are the only nontrivial kernel elements for $L$ at the round ball in $\mathbb R^n$ in  Corollary~\ref{co:star}.

\qed
\end{example}

 For later references, we denote the subspace of vector fields by 
\begin{align} 
\mathcal X^{k+1,\alpha} (\Omega)&= \big\{ X\in \C^{k+1,\alpha}(\Omega): X|_\Sigma=0\big\}  \label{eq:vector}.
\end{align}
For $(\Sigma,\bar g^\intercal) = (S^{n-1}, g_{S^{n-1}})$, we define
\begin{align}
\mathcal Z^{k+1,\alpha} (\Omega) &=\big\{ X\in \C^{k+1,\alpha}(\Omega): X|_\Sigma\in \Span\{ Z_1, \dots, Z_n\}\big\}\label{eq:Zvector},
\end{align}
where  $Z_i =\left(\tfrac{\partial}{ \partial x_i}\right)^\intercal$ denotes the conformal Killing vector field as above. 

\subsection{The Green-type identity}\label{se:green}

The classical Green identity states that the Laplace operator is self-adjoint with respect to either the Dirichlet or Neumann boundary condition. 
Here, we establish a Green-type identity for an operator $P(h)$, essentially the linearized Einstein operator,  with the Anderson boundary data in Proposition~\ref{pr:Green-conf} below and derive other useful identities. 

For a fixed real number $\Lambda$, consider the Einstein-Hilbert functional 
\[
	\mathcal F_\Lambda(g) = \int_\Omega (R_g - (n-1)(n-2) \Lambda) \dvol_g. 
\]
Let $g(s, t)$ be a smooth two-parameter family of Riemannian metrics on $\Omega$ with $g(0, 0) = g$, $\left.\ps\right|_{s=t=0} g(s,t) = h$, $\left.\pt\right|_{s=t=0} g(s,t) = w$. We calculate 
\begin{align*}
	&\left.\pt\right|_{t=0}	\mathcal F_\Lambda(g(s,t))\\
	&  =  \int_{\Omega} \Big\langle  - \Ric_{g(s, 0)} +\tfrac{1}{2} \big(R_{g(s, 0)} - (n-1)(n-2) \Lambda \big) g(s,0),  \left.\pt \right|_{t=0} g(s,t)\Big\rangle_{g(s,0)}  \dvol_{g(s,0)} \\
	&\quad - \int_\Sigma \Big\langle (A_{g(s, 0)}, 2),  \left.\pt \right|_{t=0}\big(g^\intercal(s,t), H_{g(s,t)}\big) \Big\rangle_{g(s,0)} \,\da_{g(s,0)}
\end{align*}
and
\begin{align*}
	&\left. \pst\right|_{s=t=0} \mathcal{F}_\Lambda (g(s,t))\\
	& =\int_{\Omega}\left(\big\langle P_{(g)}(h),  w\big \rangle_g  +  \Big\langle - \Ric_g +\tfrac{1}{2} \big(R_g - (n-1)(n-2) \Lambda \big) g, \left. \pst\right|_{s=t=0}g(s,t)\Big\rangle_g\right) \dvol_g\\
	&\quad -\int_\Sigma  \bigg(\Big\langle  \big(\tilde Q_{(g)}(h) - (2A_g\circ h^\intercal, 0)\big), \big(w^\intercal, H'|_g(w)\big)\Big\rangle_g\da_g\\
	&\quad -\int_\Sigma \Big\langle \big(A_g , 2\big),  \left. \pst\right|_{s=t=0}\big(g^\intercal(s,t), H_{g(s,t)}\big) \Big\rangle_g \bigg) \da_g
\end{align*}
where $P_{(g)}, \tilde Q_{(g)}$ are the linear differential operators defined by
\begin{align}\label{eq:PQ}
\begin{split}
	P_{(g)}(h) &= -\Ric'|_g(h) + \tfrac{1}{2} R'|_g(h) g + \tfrac{1}{2} \big(R_g-  (n-1)(n-2)\Lambda\big) h\\
	&\quad   - 2\Big( - \Ric_g +\tfrac{1}{2} \big(R_g - (n-1)(n-2) \Lambda \big) g\Big)\circ h  \\
	&\quad + \tfrac{1}{2} (\tr_g h)  \big( - \Ric_g +\tfrac{1}{2} \big(R_g - (n-1)(n-2) \Lambda \big) g \big) \quad \quad  \mbox{ in } \Omega 
	\\
	\tilde Q_{(g)}(h)&= \Big( A'|_g(h) +\tfrac{1}{2} (\tr_g h^\intercal )A_g , \,  \tr_g h^\intercal \Big) \quad \quad \mbox{ on } \Sigma.
\end{split}
\end{align}
Since the second variation is symmetric with respect to $s$ and $t$, we obtain the following identity. (See, e.g. \cite[Proposition 3.3]{An-Huang:2022}.) 
\begin{lemma}[Green-type identity]\label{le:Green}
For any symmetric $(0,2)$-tensors $h$ and $w$ on $(\Omega, g)$, 
\begin{align}\label{eq:Green}
\begin{split}
	&\int_\Omega \big\langle P_{(g)}(h), w\big \rangle_g  \dvol_g - \int_\Omega \big\langle P_{(g)}(w),  h \big \rangle_g \dvol_g \\
	&= \int_\Sigma \big\langle \tilde Q_{(g)}(h), \big(w^\intercal, H'|_g(w) \big) \big\rangle_g\da_g - \int_\Sigma \big \langle\tilde Q_{(g)}(w)\, (h^\intercal, H'|_g(h)) \big\rangle_g  \da_g.
\end{split}
\end{align}
\end{lemma}

Rearranging the terms allows us to involve the (linearized) Anderson boundary data. 

\begin{proposition}[Green-type identity for the Anderson boundary data]\label{pr:Green-conf}
 For any symmetric $(0,2)$-tensors $h$ and $w$ on $(\Omega, g)$, 
\begin{align*}
\begin{split}
	&\int_\Omega \big\langle P_{(g)}(h), w\big \rangle_g  \dvol_g - \int_\Omega \big\langle P_{(g)}(w),  h \big \rangle_g \dvol_g \\
	&= \int_\Sigma \left\langle Q_{(g)}(h), \big(w^\intercal - \tfrac{1}{n-1} (\tr_g w^\intercal) g^\intercal, H'|_g(w) \big) \right\rangle_g\da_g \\
	&\quad - \int_\Sigma \left \langle Q_{(g)}(w), \big(h^\intercal - \tfrac{1}{n-1} (\tr_g h^\intercal)  g^\intercal , H'|_g(h)\big ) \right\rangle_g  \da_g,
\end{split}
\end{align*}
where $P_{(g)}$ is defined the same as in \eqref{eq:PQ}, and $Q_{(g)}$ is given by
\begin{align*}
	Q_{(g)}(h) &=\left( A'|_g(h) + \big(\tfrac{1}{2} - \tfrac{1}{n-1} \big) (\tr_g h^\intercal) A_g, \big( 1-\tfrac{1}{n-1} \big)\tr_g h^\intercal\right).
\end{align*}
\end{proposition}
\begin{proof}
We calculate 
\begin{align*}
\Big\langle \tilde Q_{(g)}(h), (w^\intercal, H'|_g(w))\Big\rangle_g 
&=\Big\langle Q_{(g)}(h), \big(w^\intercal -\tfrac{1}{n-1} (\tr w^\intercal) g^\intercal, H'|_g(w)\big)\Big\rangle_g\\
&\quad  + (\tfrac{1}{2}-\tfrac{1}{n-1})\tfrac{1}{n-1} (\tr_g w^\intercal) (\tr_g h^\intercal) H_g  \\
&\quad + \tfrac{1}{n-1}(\tr_g h^\intercal ) H'|_g(w)+ \tfrac{1}{n-1}( \tr_g w^\intercal) H'|_g(h) \\
&\quad  + \tfrac{1}{n-1} (\tr_g h^\intercal) A_g\cdot w^\intercal+\tfrac{1}{n-1}( \tr_g w^\intercal)A_g\cdot h^\intercal
\end{align*}
where we use $\tr_g A'|_g(h) = H'|_g(h) + A_g\cdot h^\intercal$. 

By switching $h$ and $w$, we obtain a similar expression for $\Big\langle \tilde Q_{(g)}(w), (h^\intercal, H'|_g(h))\Big\rangle_g$. Because the last three lines above are symmetric in $h$ and $w$, they are canceled out when inserted into \eqref{eq:Green}.

\end{proof}

We discuss the consequences of Proposition~\ref{pr:Green-conf} when letting $g$ be an Einstein metric $\bar{g}$ satisfying $\Ric_{\bar{g}} = (n-1)\Lambda \bar{g}$. In what follows, we omit the subscript $\bar{g}$ when considering linearization and geometric quantities (e.g., trace, inner product, volume forms) with respect to the background $\bar{g}$. In this case, the expression for $P$ is simplified as below (we also include the definition of $Q(h)$ for easier reference later)
\begin{align}\label{eq:Green-conf2}
\begin{split}
	P(h) &=  -\Ric'(h) + \tfrac{1}{2} R'(h) \bar g + (n-1)\Lambda h \\		
	Q(h) &=\left( A'(h) + \big(\tfrac{1}{2} - \tfrac{1}{n-1} \big)\tr h^\intercal A, \big( 1-\tfrac{1}{n-1} \big)\tr h^\intercal\right).
\end{split}
\end{align}
Note that $P(h)$ is algebraically equivalent to the linearized Ricci operator. Specifically, 
\[
\Ric'(h) - (n-1) \Lambda h =-P(h)+\tfrac{1}{n-2}\big(\tr P (h)\big)\bar g.
\] 
Or equivalently, if we denote the left hand by $K(h) = \Ric'(h) - (n-1) \Lambda h$, then 
\begin{align}\label{eq:K}
	P(h) = - K(h) + \tfrac{1}{2} \big(\tr K (h)\big) \bar g.
\end{align}
It is straightforward to check that the ranges of $P$ and $K$ are isomorphic.

\begin{corollary}\label{co:Green-kernel}
 For any symmetric $(0,2)$-tensors $h$ and $w$ on $(\Omega, \bar g)$, 
\begin{align}\label{eq:Green-cor}
\begin{split}
	&\int_\Omega \big\langle P(h), w\big \rangle  \dvol - \int_\Omega \big\langle P(w),  h \big \rangle \dvol \\
	&= \int_\Sigma \left\langle Q(h), \big(w^\intercal - \tfrac{1}{n-1} (\tr w^\intercal) \bar g^\intercal, H'(w) \big) \right\rangle\da \\
	&\quad - \int_\Sigma \left \langle Q(w) ,  \big(h^\intercal - \tfrac{1}{n-1} (\tr h^\intercal)  \bar g^\intercal , H'(h)\big ) \right\rangle  \da
\end{split}
\end{align}
where $P, Q$ are defined by \eqref{eq:Green-conf2}.
\end{corollary}

Recall the operator $L$ defined in \eqref{eq:L}. We consider $h\in \Ker L$ and find ``test'' $w$ satisfying $\Ric'(w) = (n-1) \Lambda w$ in \eqref{eq:Green-cor} to obtain additional boundary conditions for $h$. The first choice of $w$ is ``infinitesimal dilation'' by letting $w=\bar g$.  

\begin{lemma}[Zero mean value when $\Lambda =0$]
\label{le:trH}
If $h$ solves $L(h)=0$, then 
\begin{align} \label{eq:trH}
	\int_\Sigma (\tr h^\intercal )H\da = -(n-1)^2 \Lambda \int_\Omega \tr h \dvol.
\end{align}
Consequently, when $\Lambda =0$, we obtain $\int_\Sigma (\tr h^\intercal )H\da =0$. 
\end{lemma}
\begin{proof}
We let $w = \bar g$. Using $\Ric'(\bar g)=0$, $P(\bar g)= -\tfrac{1}{2} (n-1)(n-2) \Lambda \bar g$ in $\Omega$, and $w^\intercal = \tfrac{1}{n-1} (\tr w^\intercal ) \bar g^\intercal$ and $ H'(w) = -\tfrac{1}{2} H$  on $\Sigma$. We compute 
\[
	\int_\Omega \langle P(h), \bar g \rangle \dvol- \int_\Omega \langle P(\bar g), h\rangle \dvol = \tfrac{1}{2} (n-1)(n-2) \Lambda \int_\Omega \tr h \dvol
\]
and the boundary terms
\begin{align*}
 &\int_\Sigma \left\langle Q(h), \big(w^\intercal - \tfrac{1}{n-1} \tr w^\intercal \bar g^\intercal, H'(w) \big) \right\rangle\da \\
 &\quad - \int_\Sigma \left \langle Q(w), \big(h^\intercal - \tfrac{1}{n-1} \tr h^\intercal  \bar g^\intercal , H'(h)\big ) \right\rangle  \da\\
 &=\int_\Sigma \left\langle Q(h), \big(0, -\tfrac{1}{2} H \big) \right\rangle\da= -\tfrac{n-2}{2(n-1)} \int_\Sigma (\tr h^\intercal) H \da. 
\end{align*}
Equating the previous two identities yields the result. 
\end{proof}

Our next choice of $w$ is ``infinitesimal diffeomorphism'' by letting $w = \mathscr L_X\bar g$ for arbitrary vector fields $X$. For our later applications, we consider $h$ which satisfies a more general linear equation as in the statement. In particular, when $b=0$, $h\in \Ker L$. 
\begin{lemma}[Hidden boundary conditions]\label{le:hidden}

Let $b\in \mathbb R$ and $h$ solve 
\begin{align*}
	&\Ric'(h)- (n-1) \Lambda h - (n-1) \Lambda b\bar g=0 \quad \mbox{ in } \Omega\\
	&\left\{\begin{array}{l}   h^\intercal = \tfrac{1}{n-1} (\tr h^\intercal) \bar g^{\intercal} \\ H'(h)=0 
	\end{array}\right. \quad \mbox{ on } \Sigma.
\end{align*}	

Then $h$ satisfies the following boundary conditions on $\Sigma$:
\begin{align}
	-2 \Div_\Sigma \left( A'(h) - \tfrac{1}{n-1} \tr h^\intercal A\right) &= \big(A-\tfrac{1}{n-1} H \bar g^\intercal\big) (\nabla_\Sigma \tr h^\intercal, \cdot)\label{eq:hiddenbd1}\\
	2A \cdot \big(A'(h) - \tfrac{1}{n-1}  \tr h^\intercal A\big) &=   -\tfrac{n-2}{n-1} L_\Sigma \tr h^\intercal + (n-1)(n-2) \Lambda b\label{eq:hiddenbd2}
\end{align}
where recall $L_\Sigma v:= -\Delta_\Sigma v - \tfrac{1}{n-2} R_\Sigma v $ as defined in \eqref{eq:bdry-op}.
\end{lemma}
\begin{proof}
Substitute $w= \mathscr L_X \bar g$ in \eqref{eq:Green-cor}  for an arbitrary vector field $X$ supported in the collar neighborhood of $\Sigma$. 
We have  $P(h) = \frac{1}{2} (n-1)(n-2) \Lambda b \bar g$, and the interior integrals become
\begin{align*}
	\int_\Omega \big \langle P(h), w  \big\rangle\dvol - \int_\Omega \big\langle P(w), h \big \rangle  \dvol &= \tfrac{1}{2} (n-1) (n-2) \Lambda b \int_\Omega \tr (\mathscr L_X \bar g) \dvol \\
	& = (n-1)(n-2) \Lambda b \int_\Sigma X\cdot \nu \da.
\end{align*}

First, let $X|_\Si$ be tangential to $\Sigma$ and apply \eqref{eq:Green-cor}:
\begin{align*}
	0&=\int_\Sigma \left\langle \big(A'(h) + \left(\tfrac{1}{2} - \tfrac{1}{n-1} \right) \tr h^\intercal A, \big(1-\tfrac{1}{n-1}\big) \tr h^\intercal \big),  \big( \mathscr L_X \bar g^\intercal - \tfrac{2}{n-1} \Div_\Sigma X \bar g^\intercal, X (H)\big )\right\rangle\da\\
	&=\int_\Sigma -2 \Div_\Sigma \left(A'(h) + \big(\tfrac{1}{2} - \tfrac{1}{n-1} \big) \tr h^\intercal A \right)\cdot X\da \\
	&\quad + \int_\Sigma \left(\tfrac{2}{n-1} \nabla_\Sigma \tr \left(A'(h)+ (\tfrac{1}{2} - \tfrac{1}{n-1} )\tr h^\intercal A\right) + \big(1-\tfrac{1}{n-1}\big) \tr h^\intercal \nabla_\Sigma H\right)\cdot X \da.
\end{align*}
Since $X$ can be an arbitrary tangential vector on $\Sigma$, we obtain~\eqref{eq:hiddenbd1} as follows:
\begin{align*}
	0&= -2 \Div_\Sigma \left(A'(h) + \left(\tfrac{1}{2} - \tfrac{1}{n-1} \right) \tr h^\intercal A \right) \\
	&\quad +\tfrac{2}{n-1} \nabla_\Sigma \tr \left(A'(h)+ (\tfrac{1}{2} - \tfrac{1}{n-1} )\tr h^\intercal A\right) +\big(1-\tfrac{1}{n-1}\big) \tr h^\intercal \nabla_\Sigma H\\
	&=-2 \Div_\Sigma \left( A'(h) - \tfrac{1}{n-1} \tr h^\intercal A \right) - \Div_\Sigma (\tr h^\intercal A) + \tfrac{1}{n-1} \nabla_\Sigma (\tr h^\intercal H) + \big(1-\tfrac{1}{n-1}\big) \tr h^\intercal \nabla_\Sigma H\\
	&=-2 \Div_\Sigma \left( A'(h) - \tfrac{1}{n-1} \tr h^\intercal A\right) - \big(A-\tfrac{1}{n-1} H \bar g^\intercal\big) (\nabla_\Sigma \tr h^\intercal, \cdot)
\end{align*}
where we use $\tr A'(h) = H'(h) + A\cdot h^\intercal = \tfrac{1}{n-1} (\tr h^\intercal) H$ and $\Div_\Sigma A=\nabla_\Sigma H$ from the Codazzi equation. 

Then let $X|_\Sigma = f \nu$ be a normal vector to $\Sigma$ and apply \eqref{eq:Green-cor}:
\begin{align*}
	 &(n-1)(n-2) \Lambda b \int_\Sigma f \da\\
	 &=\int_\Sigma \left\langle \big(A'(h) + \left(\tfrac{1}{2} - \tfrac{1}{n-1} \right) \tr h^\intercal A,\big(1-\tfrac{1}{n-1}\big) \tr h^\intercal \big), ( 2f A - \tfrac{2}{n-1} fH \bar g^\intercal,-\Delta_\Sigma f - |A|^2 f -\Ric(\nu,\nu)f)\right \rangle\da\\
	&= \int_\Sigma \left (2A \cdot A'(h) + 2\big(\tfrac{1}{2} - \tfrac{1}{n-1} \big) \tr h^\intercal |A|^2\right) f \da \\
	&\quad +\int_\Sigma \left( - \tfrac{2}{n-1} H \tr \big(A'(h) + \big(\tfrac{1}{2} - \tfrac{1}{n-1} \big)  \tr h^\intercal A\big) - \tfrac{n-2}{n-1} \Delta_\Sigma \tr h^\intercal - \tfrac{n-2}{n-1} (|A|^2+(n-1)\Lambda)\tr h^\intercal \right) f\da.
\end{align*}
Since $f$ is arbitrary, we obtain 
\begin{align*}
	&(n-1)(n-2) \Lambda b&\\
	&=2A \cdot A'(h) -  \tfrac{2}{n-1}  \tr h^\intercal |A|^2 - \tfrac{1}{n-1} (H^2-|A|^2+(n-1)(n-2)\Lambda) \tr h^\intercal  - \tfrac{n-2}{n-1} \Delta_\Sigma \tr h^\intercal\\
	& =2A \cdot \big(A'(h) - \tfrac{1}{n-1}  \tr h^\intercal A\big) -   \tfrac{n-2}{n-1} \Delta_\Sigma \tr h^\intercal - \tfrac{1}{n-1} R_\Sigma \tr h^\intercal.
\end{align*}
where we use the Gauss equation $R_\Sigma=H^2-|A|^2+(n-1)(n-2)\Lambda$. It gives \eqref{eq:hiddenbd2}.

\end{proof}

\subsection{Harmonic gauge and Fredholm alternative}\label{se:gauge}
Modify the operator $T$ defined in \eqref{eq:T} using the \emph{harmonic gauge}:
\begin{align}
	&T^{\Ga}:\mathcal M^{k,\alpha} (\Omega) \to \C^{k-2,\alpha}(\Omega)\times \C^{k-1,\alpha}(\Sigma) \times  \mathcal S_1^{k,\alpha}(\Sigma)\times \C^{k-1,\alpha}(\Sigma)\notag\\
\begin{split} \label{eq:TG}
	&T^{\Ga}(g ) = \left\{ \begin{array}{ll} \Ric_{g}  - (n-1)\Lambda g+ \mathcal D_g  \beta_{\bar g} g\quad \mbox{ in } \Omega \\
	\left\{ \begin{array}{l} \beta_{\bar g} g \\   |g^{\intercal}|^{-\frac{1}{n-1}} g^\intercal\\ H_g 
	\end{array}\right. \quad \mbox{ on } \Sigma
	\end{array} \right.
\end{split}
\end{align}
where for any vector field $X$ and symmetric $2$-tensor $h$, we denote the Killing operator $\mathcal D_g$ and the Bianchi operator $\beta_g$ as 
\begin{align*} 
	\mathcal D_g X = \tfrac{1}{2} \mathscr L_X g \quad\mbox{ and } \quad \beta_{g} h = - \Div_{g} h + \tfrac{1}{2} d \tr_{g} h.
\end{align*}
The first copy of $\C^{k-1,\alpha}(\Sigma)$ in the codomain of $T^\Ga$ denotes the space of 1-forms on the tangent bundle of $\Omega$ along $\Sigma$ (so it acts on all normal and tangential vectors to $\Sigma$).  The linearized operator of $T^{\Ga}$ at $\bar g$ is given by
\begin{align} \label{eq:LG}
 &L^\Ga: \C^{k,\alpha}(\Omega)\to\C^{k-2,\alpha}(\Omega)\times \C^{k-1,\alpha}(\Sigma)\times T\mathcal S_1^{k,\alpha}  (\Sigma)\times \C^{k-1,\alpha}(\Sigma)\notag\\
 \begin{split}
	&L^{\Ga}(h ) = \left\{ \begin{array}{ll} \Ric'(h)  -(n-1)\Lambda h + \mathcal D_{\bar g}  \beta_{\bar g} h\quad \mbox{ in } \Omega  \\
	\left\{ \begin{array}{l} \beta_{\bar g} h \\   
	h^\intercal - \tfrac{1}{n-1} (\tr h^\intercal) \bar g^{\intercal} \\ H'(h) 
	\end{array}\right. \quad \mbox{ on } \Sigma. 
	\end{array} \right.
\end{split}
\end{align}
Anderson showed that $L^\Ga$ is elliptic with Fredholm index $0$ in~\cite[Theorem 1.2]{Anderson:2008-banach} and  \cite[Proposition 3.1]{Anderson:2008}.

Consider the operator $B$ from \eqref{eq:example0} by letting $f = (n-1) \Lambda$: For vector fields $Y$,
\begin{align}\label{eq:B}
	B Y = \left\{ \begin{array}{ll} \Delta_{\bar g} Y + (n-1)\Lambda Y  & \mbox{ in } \Omega\\
	Y & \mbox{ on } \Sigma \end{array}\right..
\end{align}
The existence of the harmonic gauge relies on the solvability of $B$. By Fredholm alternative, $B$ is an isomorphism if and only if its kernel is trivial. For $\Lambda \le 0$, the kernel is always trivial. However, when $\Lambda > 0$, we make the following no-kernel assumption: 
\begin{align}\label{eq:NK}
\text{Ker}\, B = \{ 0 \}.
\end{align}
In light of Lemma~\ref{le:motivation}, the assumption~\eqref{eq:NK} holds for generic $\bar g$.

The next lemma is a standard fact that asserts solving $T(g) = (0, \gamma, \phi)$ can be reduced to solving the elliptic problem $T^\Ga(g) = (0, 0, \gamma, \phi)$. We include the proof to explain where we rely on the assumption~\eqref{eq:NK}.

Recall $\mathscr D^{k+1,\alpha} (\Omega)$  is the space of $\C^{k+1, \alpha}$-diffeomorphisms of $\Omega$ whose restriction on $\Sigma$ is the identity map $\mathrm{Id}_\Sigma$. The tangent space of $\mathscr D^{k+1,\alpha}  (\Omega)$ at the identity map $\mathrm{Id}_\Omega$ is given by $\mathcal X^{k+1,\alpha} (\Omega)$ in \eqref{eq:vector}.

\begin{lemma}\label{le:gauge}
Fix $\Lambda$ and let $\bar g\in \mathcal M^{k,\alpha}_\Lambda(\Omega)$. In the case where $\Lambda>0$, we additionally make the no-kernel assumption \eqref{eq:NK}. Then there exists a neighborhood $\mathcal U$ of $\bar g$ in $\mathcal M^{k,\alpha}(\Omega)$ and a neighborhood $\mathscr D_0$ of the identity map $\mathrm{Id}_\Omega$ in $\mathscr D^{k+1,\alpha}(\Omega)$ such that the following holds:
\begin{enumerate}
\item \label{it:gauge0} If $g\in \mathcal U$ solves $T^{\Ga}(g) = (0, 0, \gamma, \phi)$, then $g$ satisfies the harmonic gauge $\beta_{\bar g} g=0$ in $\Omega$, and thus $T(g) = (0, \gamma, \phi)$.
\item For any $g\in \mathcal U$, there exists a unique diffeomorphism $\psi\in \mathscr D_0$ such that $\beta_{\bar g} (\psi^* g)=0$ in $\Omega$, and $\mathscr D_0$ acts on $\mathcal U$ freely; that is, $g = \psi^* g$  implies $\psi=\mathrm{Id}_\Omega$. As a direct consequence, if $g\in \mathcal U$ solves $T(g)=(0, \gamma, \phi)$, then there exists a unique $\psi\in \mathscr D_0$ such that $T^\Ga(\psi^* g) = (0, 0, \gamma, \phi)$.\label{it:gauge}
\end{enumerate}
\end{lemma}
\begin{proof}
To prove the first statement, let $g$ solve $T^{\Ga}(g) = (0, 0, \gamma, \phi)$. Denote $Y= \beta_{\bar g} g$. Apply $\beta_g$ on the first equation of $T^\Ga(g)$  gives 
\begin{align*}
	0&=\beta_g (\Ric_g -(n-1)\Lambda g) + \beta_g \mathcal D_g Y= \beta_g \mathcal D_g  Y= - \tfrac{1}{2}( \Delta_g Y + \Ric_{g}(Y, \cdot) ),
\end{align*}
where we use \eqref{eq:Laplace0}. At $\bar g$, the operator $\Delta_{\bar g} Y + (n-1)\Lambda Y$ is an isomorphism, so is the operator  $\Delta_g Y + \Ric_{g}(Y, \cdot)$ for $g$ sufficiently close to $\bar g$. Hence $Y\equiv 0$ in $\Omega$. 

For the second statement, we define a map $f: \mathscr D^{k+1,\alpha} (\Omega) \times \mathcal M^{k,\alpha}  (\Omega) \to \C^{k-1,\alpha}(\Omega)$ by $f(\psi, g) = \beta_{\bar g} (\psi^* g)$. Linearizing with respect to the first argument at $(\mathrm{Id}_\Omega, \bar g)$ gives $D_1 f : \mathcal X^{k+1, \alpha} (\Omega) \to \C^{k-1,\alpha}(\Omega)$ by $D_1f (Y)= -\tfrac{1}{2} (\Delta_{\bar g} Y+ (n-1)\Lambda Y)$.  Since $D_1f$ is an isomorphism, by the Inverse Function Theorem, there exist $\mathcal{U}$ and $\mathscr{D}_0$ such that for each $g\in \mathcal{U}$, there exists a unique $\psi\in \mathscr{D}_0$ such that $\beta_{\bar{g}} (\psi^* g) = 0$. The fact that $\mathscr{D}_0$ acts on $\mathcal{U}$ freely follows from the fact that if $g = \psi^* g$ satisfies the harmonic gauge, then $\psi = \mathrm{Id}_\Omega$ from the above application of the Inverse Function Theorem.
\end{proof}

A corresponding statement for the linearized operators also hold. 
\begin{lemma}\label{le:gauge-linear}
Fix $\Lambda$ and let $\bar g\in \mathcal M^{k,\alpha}_\Lambda(\Omega)$. In the case where $\Lambda>0$, we additionally  make the no-kernel assumption \eqref{eq:NK}. Then the following holds:
\begin{enumerate}
\item If $h\in \Ker L^\Ga $, then $\beta_{\bar g} h =0$ and $h\in \Ker L$.
\item If $h\in \Ker L$, then there exists a unique $Y\in \mathcal X^{k+1,\alpha}(\Omega)$ such that $\beta_{\bar g} (h+ \mathscr L_Y \bar g)=0$ and  $h+ \mathscr L_Y\bar g\in \Ker L^\Ga$. 
\end{enumerate}
\end{lemma}
\begin{proof}
To prove Item (1), we suppose $L^\Ga (h)=0$ and denoting $W =\beta h$. Applying the Bianchi operator $\beta$ on the first equation of $L^\Ga(h)$ gives
\[
	0=\beta (\Ric_{\bar g} - (n-1)\Lambda \bar g) + \beta \mathcal D W = \beta \mathcal D W = -\tfrac{1}{2} (\Delta W + (n-1)\Lambda W).
\]	
Since $W=0$ on $\Sigma$, we obtain $W\equiv 0$ in $\Omega$. Item~(2) is obvious.

\end{proof}

\begin{remark}\label{re:harmonic}
By Item~\eqref{it:gauge} of Lemma~\ref{le:gauge}, we can identify $\mathcal M^{k,\alpha}_\Lambda(\Omega) / \mathscr D^{k+1,\alpha} (\Omega)$ locally with the space of metrics that satisfy the harmonic gauge. More specifically, for $\bar g\in \mathcal M^{k,\alpha}_\Lambda(\Omega)$, let $\mathcal U$ be the neighborhood from the previous lemma and let $\mathcal U_\Lambda=\mathcal M^{k,\alpha}_\Lambda(\Omega)\cap\, \mathcal U$. Then we can identify the equivalent class $\llbracket g\rrbracket$ in $\mathcal U_\Lambda/\mathscr D^{k+1,\alpha}(\Omega)$ by taking the unique ``harmonic representative'':
\[
	\mathcal U_\Lambda/\mathscr D^{k+1,\alpha}(\Omega) = \big\{ g \in \mathcal U:  \beta_{\bar g} g=0 \mbox{ and }\Ric_g = (n-1)\Lambda g \mbox{ in  } \Omega \big\}.
\] 
Furthermore, the moduli space of negative Einstein metrics $\mathcal M^{k,\alpha}_-(\Omega) / \mathscr D^{k+1,\alpha} (\Omega)$ is also a smooth Banach manifold. To see this, denoting $\mathcal U_-=\mathcal M^{k,\alpha}_-(\Omega)\cap\, \mathcal U$, we can identify  $\mathcal M^{k,\alpha}_-(\Omega) / \mathscr D^{k+1,\alpha} (\Omega)$ locally with
\begin{align*}
	\mathcal U_-/\mathscr D^{k+1,\alpha}(\Omega) &= \big\{ g \in \mathcal U:  \beta_{\bar g} g=0 \mbox{ and }\Ric_g = \lambda g \mbox{ in  } \Omega \mbox{ for some } \lambda <0\big\}.
\end{align*} 
By identifying $g\in \mathcal U_-/\mathscr D^{k+1,\alpha}(\Omega)$  with $(\lambda g, \lambda^{-1})$, we see that $\mathcal U_-/\mathscr D^{k+1,\alpha}(\Omega)$ is locally diffeomorphic to $\left( \mathcal U_{-1}/\mathscr D^{k+1,\alpha}(\Omega) \right)\times (-\infty , 0)$ (note that $ \mathcal U_{-1}$ denotes $\mathcal U_\Lambda$ where $\Lambda = -1$) and hence is a smooth Banach manifold. 
\end{remark}

We conclude this section by illustrating how the triviality of the kernel of $L$ (in the sense of Example 2.1) leads to Theorem~\ref{th:generic0} and Theorem~\ref{th:generic-higher}.

\begin{definition}\label{de:rigid}
Fix $\Lambda$. We say that $(\Omega, \bar g)$ is \emph{infinitesimally rigid with respect to the Anderson boundary data} (among $\mathcal M_\Lambda^{k,\alpha} (\Omega)$) if
\begin{align*}
\Ker L = \big\{ \mathscr L_X \bar g: X\in \mathcal X^{k+1,\alpha}(\Omega) \big\}.
\end{align*}
\end{definition}

\begin{proposition}\label{pr:Fredholm}
Suppose $(\Omega, \bar g)$ is infinitesimally rigid with respect to the Anderson boundary data. In the case where $\Lambda>0$, we additionally make the no-kernel assumption \eqref{eq:NK}.  Then the boundary map $\Pi: \mathcal M^{k,\alpha}_\Lambda(\Omega) / \mathscr D^{k+1,\alpha}(\Omega) \to \mathcal S_1^{k,\alpha}(\Sigma)\times \C^{k-1,\alpha}(\Sigma)$ is a local diffeomorphism at $\llbracket \bar g \rrbracket$.
\end{proposition}
\begin{proof}
By the infinitesimal rigidity assumption and Lemma~\ref{le:gauge-linear},  $\Ker L^\Ga=\{ 0 \}$ (for $\Lambda>0$, we also use \eqref{eq:NK}). Because $L^\Ga$ is Fredholm of index zero, $L^\Ga$ is an isomorphism and thus by   $T^\Ga$ is a local diffeomorphism at $\bar g$ by the Inverse Function Theorem. Specifically, there exists a neighborhood $\mathcal U \subset \mathcal M^{k,\alpha}(\Omega)$ of $\bar g$ and a neighborhood of $T^\Ga(\bar g)$, denoted by 
\[
\mathcal W_1\times \mathcal W_2 \subset \Big(\C^{k-2,\alpha}(\Omega)\times \C^{k-1,\alpha}(\Sigma)\Big) \times \Big( \mathcal S_1^{k,\alpha}(\Sigma)\times \C^{k-1,\alpha}(\Sigma)\Big),
\]
 such that $T^\Ga: \mathcal U \to  \mathcal W_1\times \mathcal W_2$ is a diffeomorphism. Therefore, the restriction of $T^\Ga$ on $\mathcal U_\Lambda$ is also a diffeomorphism onto its image.   By Remark~\ref{re:harmonic},  the restriction descends to a well-defined map from $\mathcal U_\Lambda/\mathscr D^{k+1,\alpha}(\Sigma)$ to  $\mathcal W_1\times \mathcal W_2$ that is a diffeomorphism onto its image $\{ (0, 0) \} \times \mathcal W_2$ (where the ``onto'' part follows Item~\eqref{it:gauge0} of Lemma~\ref{le:gauge}). Thus, the boundary map $\Pi$ is a diffeomorphism from $\mathcal U_\Lambda/\mathscr D^{k+1,\alpha}(\Sigma)$ to  $\mathcal W_2$.

\end{proof}


\section{Conformal Cauchy boundary condition}\label{se:conf}

In this section, we analyze the infinitesimal rigidity (also referred to as unique continuation) of the conformal Cauchy boundary condition, as well as non-rigidity in the case of round spheres. We will assume that $(\Omega, \bar{g})$ satisfies $\pi_1(\Omega, \Sigma)=0$ from now on. Furthermore, only in this section,  we assume that $(\Omega, \bar g)$ is contained in a larger Einstein manifold. .

\begin{notation} \label{notation}
Let $(M, \bar{g})$ be a complete compact Riemannian $n$-dimensional manifold with smooth boundary, and $\Ric_{\bar{g}} = (n-1) \Lambda \bar{g}$. Consider $\Omega$ as a connected, proper compact subset  in $M$ whose boundary $\Sigma$ is a smooth embedded hypersurface in $\Int M$, and $\pi_1(\Omega, \Sigma) = 0$.
\end{notation}

\begin{definition}
Let $h$ be a symmetric $(0,2)$-tensor on $(\Omega, \bar{g})$. We say that $h$ satisfies the \emph{conformal Cauchy boundary condition} on $\Sigma$ if 
\begin{align*}
	h^\intercal &= \tfrac{1}{n-1}(\tr h^\intercal) \bar{g}^\intercal, \\
	A'(h) &= \tfrac{1}{n-1}(\tr h^\intercal) A.
\end{align*}
If, additionally, $\tr h^\intercal=0$, then we say $h$ satisfies the \emph{Cauchy boundary condition}. 
\end{definition}

Note that the conformal Cauchy boundary condition implies the Anderson boundary condition because $H'(h) = -h\cdot A + \tr A'(h) = 0$. We also remark that $h^\intercal=0$ and $A'(h)=0$ are referred to as the Cauchy boundary condition because, in the geodesic gauge (i.e., $h(\nu, \cdot )=0$), those boundary conditions become the classical Cauchy boundary condition $h=0$ and $\nabla_\nu h=0$ on $\Sigma$.

Before presenting the main result of this section, Theorem~\ref{th:infinitesimal} below, we first discuss the infinitesimal rigidity for the Cauchy boundary condition. For the following results, we recall the spaces of vector fields $\mathcal{X}^{k+1,\alpha} (\Omega)$ and $\mathcal{Z}^{k+1,\alpha}(\Omega)$ as defined in~\eqref{eq:vector} and \eqref{eq:Zvector}, respectively. The next result can be thought of as an infinitesimal version of the unique continuation for Einstein metrics of \cite{Anderson-Herzlich:2008,Anderson-Herzlich:erratum, Biquard:2008} under the additional assumption that the metric admits an Einstein metric extension, by a different argument.

\begin{proposition}[Infinitesimal rigidity for Cauchy boundary condition]\label{pr:zero}
Let $(\Omega, \bar g)$ be a compact Riemannian manifold as in Notation~\ref{notation}.  Let $h\in \C^{k,\alpha}(\Omega)$ satisfy 
\begin{align*}
\begin{split}
	&\Ric'(h)- (n-1) \Lambda h=0\quad  \mbox{ in $\Omega$}\\
	&\left\{ \begin{array}{l} h^\intercal = 0\\
	A'(h)= 0 \end{array} \right. \mbox{ on } \Sigma.
\end{split}
\end{align*}
Then $h=\mathscr L_X\bar g$ in $\Omega$ for some $X \in \mathcal X^{k+1,\alpha} (\Omega)$. 
\end{proposition}
\begin{proof}
We may without loss of generality, assume that $h$ is in the geodesic gauge, i.e., $h(\nu, \cdot)=0$ on $\Sigma$,  and thus the Cauchy boundary condition becomes
\begin{align}\label{eq:Cauchy2}
\begin{split}
	h= 0 \quad \mbox{ and } \quad \nabla_\nu h= 0 \quad \mbox{ on } \Sigma.
\end{split}
\end{align}
See, e.g. \cite[Lemma 2.5 and Remark 2.6]{An-Huang:2022}. We extend $h$  to the entire $M$ by defining $\hat h\in \C^{1}(M)$ as
\begin{align*}
	\hat h &= h\quad \mbox{ in } \Omega\\
	\hat h &=0 \quad \mbox{ in } M\setminus  \Omega.
\end{align*}

The steps  of our proof are as follows. 
\begin{enumerate}
\item \label{it:step1} We first ``correct'' $\hat h$ on $M$ by adding $\mathscr L_Y \bar g$ so that $w:=\hat h+\mathscr L_Y \bar g$ weakly satisfies the harmonic gauge. 
\item  \label{it:step2} We show that $w$ satisfies an elliptic equation and is thus analytic  by elliptic regularity. 
\item \label{it:step3} Using the fact that $w=\mathscr L_Y\bar g$  in $M\setminus \Omega$ and analytic continuation  (see, e.g. \cite[Theorem 9]{An-Huang:2024}), we conclude $w \equiv \mathscr L_{\hat Y} \bar g$  for some vector field $\hat Y$ on the entire $\Omega$. Then it implies the desired statement for $h$. 
\end{enumerate}

Recall the Bianchi operator $\beta_{\bar g}$ defined in \eqref{eq:KB}. Below we will  omit the subscript $\bar g$  when the geometric operators are evaluated at $\bar g$. For Step~\eqref{it:step1}, let $Y\in \C^{1,\alpha}(M)$ solve $\beta (\mathscr L_Y \bar{g} )= -\beta \hat h$, i.e. 
\[
- \Delta Y - (n-1) \Lambda Y = - \beta \hat h \quad \mbox{ in } M.
\]
(Note that the boundary value $Y|_{\partial M}$ is irreverent in this proof.)
If $\Lambda\le 0$, we let $Y$ be the unique weak solution with zero boundary value. If $\Lambda>0$, the operator $Y\mapsto \beta (\mathscr L_Y \bar{g} )$ is Fredholm and has finite-dimensional cokernel; thus the solution $Y$ exists for some boundary values. This shows that $\beta (\hat h+\mathscr L_Y \bar{g} )= 0$ weakly.  Denote $w:=\hat h+\mathscr L_Y\bar{g}$.

We proceed to prove Step~\eqref{it:step2}. We claim that $w$ weakly solves
\begin{align} \label{eq:elliptic}
	( \Delta w)_{ij} + 2 R_{ik\ell j} w^{k\ell} =0 \quad \mbox{ in }M
\end{align}
where $R_{ik\ell j}$ is the Riemann curvature tensor of $\bar g$. 
We will show that $w$ weakly solves the linearized Einstein equation:
\begin{align}\label{eq:weak}
	\Ric'(w) = (n-1) \Lambda w  \quad \mbox{ in } M.
\end{align} 
In fact, we show that both $\hat h$ and  $\mathscr L_Y \bar{g} $  weakly solve the linearized Einstein equation as follows:  For any $\gamma\in \C^\infty_c (M)$, we use integration by parts  in $\Omega$ and the boundary condition \eqref{eq:Cauchy2} to obtain
\begin{align*}
	\int_{M}  \hat h \cdot \big((\Ric')^* (\gamma) - (n-1)\Lambda \gamma\big) \dvol&= \int_{ \Omega} h\cdot \big((\Ric')^* (\gamma) - (n-1)\Lambda \gamma\big)) \dvol\\
	&=  \int_{\Omega}  \big(\Ric' (h) - (n-1) \Lambda h \big)\cdot \gamma \dvol=0.
\end{align*}
 For $\mathscr L_Y\bar{g}$, we integrate by parts and get
 \begin{align*}
 	&\int_{M} \mathscr L_{Y} \bar{g}\cdot \big( (\Ric')^* (\gamma) - (n-1)\Lambda \gamma\big)\dvol\\
	&=- \int_{M} 2Y \cdot  \Div \big((\Ric')^* (\gamma) - (n-1)\Lambda \gamma\big)\,\dvol=0
 \end{align*}
where we use \eqref{eq:divergence} in the last equality.  This proves \eqref{eq:weak}. To prove  \eqref{eq:elliptic}, we compute, for any $\gamma\in \C^\infty_c (M)$,
\begin{align*}
	0 
	&=\int_{M} w\cdot \big((\Ric')^* (\gamma)  - (n-1) \Lambda \gamma \big)\dvol \\
	&= \int_{M}  w\cdot (-\tfrac{1}{2} \Delta \gamma + \beta^* \Div \gamma -R_{ik\ell j} \gamma^{k\ell}  ) \dvol\\
	&=\int_{M} w\cdot \left(-\tfrac{1}{2} \Delta \gamma -R_{ik\ell j} \gamma^{k\ell} \right)\dvol
	\end{align*}
	where we use  \eqref{eq:liEin} and $\beta_{\bar g} w= 0$ weakly.

We finish with Step \eqref{it:step3}. Since $\bar g$ is analytic in the chosen atlas (because $\bar g$ is Einstein), $w$ is also analytic in $\Int M$  by \eqref{eq:elliptic} and elliptic regularity.  Let $U$ be an open collar neighborhood of $\Omega$ such that $\Omega \subset U \subset \Int M$ and $\pi_1( U, \Sigma)=0$ (here we use the assumption $\pi_1(\Omega, \Sigma)=0$). Since $w=\mathscr L_Y \bar g$ in the open set $U\setminus\Omega$,  by analytic continuation (see, e.g. \cite[Theorem 9]{An-Huang:2024}),  there is an analytic vector field $\widehat Y$ defined in $U$ such that $\widehat Y= Y$ in $U\setminus \Omega$ and $w = \mathscr L_{\widehat Y} \bar g$ everywhere in $U$. Based on how $w$ and $\hat h$ are constructed, by letting $X = \widehat Y - Y$ we conclude that $X|_\Sigma =0$ and $h = \mathscr L_X \bar g$ in $\Omega$. It completes the proof. 
\end{proof}

\subsection{Infinitesimal rigidity and non-rigidity}

For the conformal Cauchy boundary condition, the infinitesimal rigidity may fail, as demonstrated in Example~\ref{ex:conformal}. Nevertheless, we show that rigidity holds under mild assumptions. In the umbilic case considered in Example~\ref{ex:conformal}, we also demonstrate that these examples are the only non-trivial kernel elements. It is worth noting that corresponding global rigidity always holds, as stated in Proposition~\ref{pr:global};  the umbilic case is a consequence of Hopf's uniqueness of constant mean curvature spheres.

\begin{theorem}\label{th:infinitesimal}
Let $(\Omega, \bar g)$ be a compact Riemannian manifold as in Notation~\ref{notation}. Let $h\in \C^{k,\alpha}(\Omega)$  satisfy
\begin{align*}
	&\Ric'(h) - (n-1) \Lambda h=0 \quad \mbox{ in $\Omega$} \\
\begin{split}	
	&\left\{ \begin{array}{l} h^\intercal = \frac{1}{n-1} (\tr h^\intercal )\bar g^\intercal \\
	A'(h)=  \frac{1}{n-1} (\tr h^\intercal ) A\end{array} \right. \mbox{ on } \Sigma.
\end{split}
\end{align*}
Then the following holds:
\begin{enumerate}
\item (Infinitesimal rigidity)\label{it:nonspherical} Suppose that either 
\begin{enumerate}
\item  $n=3$, $\Sigma$ is not umbilic and $R_\Sigma$ is not identically zero, or 
\item  $n\ge 3$, all eigenvalues of $L_\Sigma$ are nonzero where $L_\Sigma$ is defined by \eqref{eq:bdry-op}
\end{enumerate} 
Then $h=\mathscr L_X \bar g$ in $\Omega$ for  some $X \in \mathcal X^{k+1,\alpha} (\Omega)$. 

\item  (Infinitesimal non-rigidity) \label{it:spherical} If $n\ge 3$ and  $(\Sigma,\bar g^\intercal)$ is a round sphere and umbilic, then  $h=\mathscr L_X\bar g$ for some $X\in \mathcal Z^{k+1,\alpha} (\Omega)$. 
\end{enumerate}
\end{theorem}
\begin{proof}

By Lemma~\ref{le:hidden}, we have the following equations on $\Sigma$:
\begin{align}
	\big(A - \tfrac{1}{n-1} H\bar g^\intercal\big)(\nabla_\Sigma \tr h^\intercal, \cdot)&=0 \notag\\
	L_\Sigma \tr h^\intercal :=- \Delta_\Sigma (\tr h^\intercal) - \tfrac{1}{n-2} R_\Sigma \tr h^\intercal&=0 \label{eq:Codazzi}.
\end{align}

We prove Item \eqref{it:nonspherical}. The assumption that $L_\Sigma$ does not have a zero eigenvalue implies $\tr h^\intercal =0$ by \eqref{eq:Codazzi}. Therefore, $h$ satisfies the Cauchy boundary condition, and the desired result follows from Proposition~\ref{pr:zero}.

For $n=3$, we obtain stronger results. Define $ \Sigma_0 := \{ p\in \Sigma: R_\Sigma (p)\neq 0\}$. We claim that $\mathring A := A - \tfrac{1}{n-1} H \bar g^\intercal $ is not identically zero on $\Sigma_0$. We may assume that $\Sigma_0$ is not the entire $\Sigma$. Otherwise, it is clear that $\mathring A$ is not identically zero on $\Sigma$ due to the assumption that $\Sigma$ is not umbilic.

Suppose, for a contradiction, that $\mathring A\equiv 0$ on $\Sigma_0$. By the Codazzi equation, the mean curvature $H$ must be constant on each component of $\Sigma_0$. The Gauss equation further implies that $R_\Sigma$ is also constant on each component of $\Sigma_0$. Thus, $R_\Sigma \equiv 0$ since $R_\Sigma=0$ at the boundary points of $\partial \Sigma_0$. This contradicts the assumption that $R_\Sigma\not \equiv 0$. Now we have shown $U= \Sigma_0\cap \{ p\in \Sigma: \mathring A(p)\neq 0\} \neq \emptyset$. When $n=3$, it implies both eigenvalues of $\mathring A$ are nonzero on $U$, and thus $\tr h^\intercal$ is constant on $U$ and thus must be zero everywhere by \eqref{eq:Codazzi}. Again, the result then follows Proposition~\ref{pr:zero}.

Now, let's discuss Item~\eqref{it:spherical}. The assumption that $(\Sigma, \bar g^\intercal)$ is a round sphere implies that solutions to $L_\Sigma v=0 $ are  spherical harmonic functions of degree $1$.  Thus, if $\tr h^\intercal \neq 0$, then $\tr h^\intercal= \Div_\Sigma Z$ for a conformal Killing vector field $Z\in \Span\{ Z_1, \dots, Z_n\}$. We smoothly extend $Z$ to be defined entirely on $\Omega$ and denote the extension by $\hat Z$. Define $w:= \mathscr L_{\hat Z} \bar g$. Then  $\Ric' (w)-(n-1)\Lambda w =0$  in $\Omega$. Since $\tr w^\intercal = \tr h^\intercal $ on  $\Sigma$, we have  $w^\intercal=h^\intercal$ and 
\begin{align*}
A'(w)&= \tfrac{1}{n-1} (\tr w^\intercal) A =   \tfrac{1}{n-1} (\tr h^\intercal) A =A'(h) 
\end{align*}
where in the first equality we use \eqref{eq:2ffK}. Therefore, $h-w$ satisfies the Cauchy boundary condition, so $ h- w = \mathscr L_X \bar g$ for some $X\in \mathcal X^{k+1,\alpha}(\Omega)$ from Proposition~\ref{pr:zero}. This implies $h = \mathscr L_{X+\hat Z} \bar g$ and $X+\hat Z\in \mathcal Z^{k+1,\alpha}(\Omega)$. 

\end{proof}

\section{Domain deformations and genericity}\label{se:generic}

In this section, we extend the outline described in Lemma~\ref{le:motivation} to our geometric operator.  The main result is  Theorem~\ref{th:generic-bdry} below, and then we will combine it with the results established in the prior sections to prove Theorems~\ref{th:generic0}-\ref{th:bdrymap}.

\subsection{Conformal Cauchy boundary condition}\label{se:deform}


\begin{notation} \label{no:foliation}
Let $\Omega$ be an $n$-dimensional compact manifold with smooth boundary $\Sigma$ and $\pi_1(\Omega, \Sigma)=0$. Let $\bar g$ be the background Einstein metric satisfying $\Ric_{\bar g} = (n-1) \Lambda \bar g$. 

Consider the \emph{inward deformation} $\{ \Omega_t\} $ for $t\in [0, \delta)$ for some $\delta>0$ whose boundaries $\{\Sigma_t \}$ form a smooth foliation of embedded hypersurfaces in $\Omega$. Denote the deformation vector $V = \left.\frac{\partial}{\partial t} \right|_{t=0}\Sigma_t$. We extend $V$ to be a smooth vector field on $\Omega$ and assume $V=-\zeta \nu_{\bar g}$ on each $\Sigma_t$ for some positive function $\zeta$.  Let $\psi_t:\Omega \to \Omega_t$  be the diffeomorphism generated by  $V$ with $\psi_0= \mathrm{Id}$. 
\end{notation}
\begin{remark}
When $n\geq 4$, we will primarily focus on two types of inward domain deformations. Setting $\zeta =1$, the family of regions $\Omega_t$ is deformed from $\Omega$ by the unit normal vector field. Setting  $\zeta = H_{\bar  g}$, the mean curvature of $\Sigma_t$ in $(\Omega, \bar g)$, the deformation is deformed by the boundary mean curvature flow.
\end{remark}

We will analyze the linearized operator $L$ at $(\Omega_t, \bar g|_{\Omega_t})$. It is more convenient to pull back $L$ from $\Omega_t$ to the fixed region $\Omega$, denoted by $L_t$,  and view the family of operators $\{ L_t \}$ acting on the same function space.  We denote the pull-back metric $g(t) = \psi_t^* (\bar g|_{\Omega_t} )$ and write the operators with respect to $g(t)$  by
\begin{align} \label{eq:Lt}
\begin{split}
&L_t: \C^{k,\alpha}(\Omega)\to\C^{k-2,\alpha}(\Omega)\times T \mathcal  S_1^{k,\alpha}  (\Sigma)\times \C^{k-1,\alpha}(\Sigma)\\
	&L_t(h ) = \left\{ \begin{array}{ll} \Ric'|_{g(t)}(h) - (n-1) \Lambda h  \quad \mbox{ in } \Omega \\
	\left\{ \begin{array}{l}   |g(t)^\intercal |^{-\frac{1}{n-1}} (h^\intercal - \tfrac{1}{n-1} (\tr_{g(t)} h^\intercal)g(t)^{\intercal}) \\ H'|_{g(t)}(h) 
	\end{array}\right. \quad \mbox{ on } \Sigma
	\end{array} \right..
\end{split}
\end{align}
The corresponding  operator in the harmonic gauge is defined by 
\begin{align*}
&L^{\Ga}_t: \C^{k,\alpha}(\Omega)\to\C^{k-2,\alpha}(\Omega)\times \C^{k-1,\alpha}(\Sigma)\times T \mathcal  S_1^{k,\alpha}  (\Sigma)\times \C^{k-1,\alpha}(\Sigma)\\
	&L^{\Ga}_t(h ) = \left\{ \begin{array}{ll} \Ric'|_{g(t)}(h) - (n-1) \Lambda h  + \mathcal D_{g(t)}  \beta_{g(t)} h\quad \mbox{ in } \Omega \\
	\left\{ \begin{array}{l}  \beta_{g(t)} h\\
	 |g(t)^\intercal |^{-\frac{1}{n-1}} (h^\intercal - \tfrac{1}{n-1} (\tr_{g(t)} h^\intercal)g(t)^{\intercal}) \\ H'|_{g(t)}(h) 
	\end{array}\right. \quad \mbox{ on } \Sigma
	\end{array} \right..
\end{align*}

Define the nullity function $N(t) =\Dim \Ker L^{\Ga}_t$. We note that since the kernel elements must be smooth by elliptic regularity, $N(t)$ is defined independently of the regularity $(k, \alpha)$.  

We show that the nullity function $N(t)$ satisfies the following properties. While a similar lemma was proven in \cite[Lemma 6.5]{An-Huang:2022}, we provide the proof for completeness; notably, the proof uses only the fact that the family of operators satisfy the uniform Schauder estimate.
\begin{lemma}\label{le:nullity}
\begin{enumerate}
\item Let $\limsup_{t\to a^+} N(t) = \ell$ for some integer $\ell$.  Then there exists a decreasing sequence $\{t_j\}$ with $t_j\searrow a$ such that $N(t_j) = \ell $ and that $\Ker L^{\Ga}_{t_j}$ converges to an $\ell$-dimensional subspace of $\Ker L^{\Ga}_a$. Consequently,  $N(t)$ is upper semi-continuous. \label{it:semi-cont}
\item Define the subset $J\subset [0, \delta)$ by 
\begin{align}\label{eq:J}
	J = \{ t\in [0, \delta): \mbox{$N(t)$ is constant in an open neighborhood of $t$} \}.
\end{align}
Then $J$ is open and dense. \label{it:J}
\end{enumerate}
\end{lemma}
\begin{proof}
We first prove a general fact. Suppose that for some positive integer $\ell$ and $t_j\to a^+$, $N(t_j)\ge \ell$ for all $j$.  Let $K_j$ be an $\ell$-dimensional subspace of $\Ker L_{t_j}^\Ga$. We claim that $K_j$ converges subsequently to an $\ell$-dimensional subspace of $\Ker L_a^{\Ga}$. To see that, we assume $K_j$ is spanned by a basis $\{ h_1(t_j), \dots, h_\ell(t_j)\}$, orthonormal with respect to $\bar g$. By Schauder estimate, $h_1(t_j)$ admits a subsequence, say labeled by $t_{j_k}$, convergent to $h_1(a)\in \Ker L^{\Ga}_a$. Then we can take a further subsequence $\{t'_{j_k}\}\subset\{t_{j_k}\}$ so that $\{ h_2(t'_{j_k})\}$ converges to $h_2(a)\in \Ker L^{\Ga}_a$ and verify that $h_2(a)$ is orthogonal to $h_1(a)$. Repeating the steps finitely many times, we obtain an $\ell$-dimensional subspace of $\Ker L^{\Ga}_a$.  

 
For Item~\eqref{it:semi-cont}, since $N(t)$ takes values on integers, there is a sequence $t_j\to a^+$ such that $N(t_j) =\ell$ for all $j$. The above argument implies that  $\ell \le N(a)$. The same inequality also holds for $t\to a^-$, and thus we conclude the upper semi-continuity of $N(t)$. Item~\eqref{it:J} follows from a general fact that an upper semicontinuous and integer-valued function is locally constant on an open dense subset.

\end{proof}

\begin{proposition}\label{pr:kernel}
Let $J$ be the open dense subset defined by \eqref{eq:J}. Then, for any $a\in J$ and for any $h\in \Ker L^{\Ga}_a$, there exists a sequence $\{ t_j \}\subset J$ with $t_j \searrow a$, $h(t_j)\in \Ker L^{\Ga}_{t_j}$, and a symmetric $(0,2)$-tensor $p\in \C^{k,\alpha}(\Omega)$ such that as $t_j \searrow a$
\[
    h(t_j) \to h \quad \mbox{ and } \quad \frac{h(t_j)- h}{t_j -a} \to p \quad \mbox{ in } \C^{k,\alpha}(\Omega).
\]

\end{proposition}
\begin{proof}
In the proof below, we denote the codomain space of $L^{\Ga}_t$ by $\mathcal B:=\C^{k-2,\alpha}(\Omega)\times \C^{k-1,\alpha}(\Sigma)\times T \mathcal  S_1^{k,\alpha}  (\Sigma)\times \C^{k-1,\alpha}(\Sigma)$.

 Since $a\in J$, there is a nonnegative integer $\ell$ such that $ N(t) = \ell$ for $t$ in an open neighborhood of $a$ in $[0,\delta)$.  By Lemma~\ref{le:nullity},  there is a sequence $t_j\searrow a$ such that $N(t_j)=\ell=N(a)$ and  $\Ker L^{\Ga}_{t_j}$ converges to $\Ker L^{\Ga}_a$. That is,  for any $h\in \Ker L^{\Ga}_a$, there is a sequence $\hat{h}(t_j)$ in $\Ker L^{\Ga}_{t_j}$ such that $\hat{h}(t_j)\to h$ in $\C^{k,\alpha}(\Omega)$.

Note the difference quotient $\frac{ \hat{h}(t_j)-h} {t_j -a}$ 
needs not converge in general.  We will  modify the sequence $\hat{h}(t_j)$ by adding a sequence of ``correction'' terms from $\Ker L^{\Ga}_{t_j}$.  To proceed, observe that since $\Ker L^{\Ga}_{t_j}$ converges to $\Ker L^{\Ga}_a$, there is a closed subspace $\mathcal{Y}$ such that, for all $j$ sufficiently large,  $\mathcal{Y}$ is the complement to both $\Ker L^{\Ga}_{t_j}$ and  $\Ker L^{\Ga}_a$. Then the following elliptic estimate holds for all $h\in \mathcal Y$ and $j$ sufficiently large:
\begin{align}\label{eq:elliptic-ker}
	\| h \|_{\C^{k,\alpha}(\Omega)} \le C \| L^{\Ga}_{t_j} h \|_{\mathcal B} \quad \mbox{ for a constant $C$ independent of $j$}.
\end{align}
Let $w(t_j)\in \Ker L^{\Ga}_{t_j}$ be the ``correction'' term that satisfies
\begin{align*}
	\frac{ \hat{h}(t_j)-h} {t_j -a} - w(t_j) \in \mathcal{Y}.
\end{align*}
Applying $L_{t_j}^\Ga$ to the previous element gives 
\begin{align*}
	&L^{\Ga}_{t_j} \left(\frac{ \hat{h}(t_j)-h} {t_j -a} - w(t_j)  \right)=-\frac{1}{t_j-a} L^{\Ga}_{t_j} h = -\frac{1}{t_j-a} (L^{\Ga}_{t_j} - L^{\Ga}_a)h.
\end{align*}
Note that $f_j:= -\frac{1}{t_j-a} \big(L^{\Ga}_{t_j} - L^{\Ga}_a\big)h$ converges to some $f$ in $\mathcal{B}$ as $j\to \infty$ because $L_{t}^\Ga$ is a differentiable family of operators.  Thus, $\| f_j \|_{\mathcal{B}}$ is bounded above by a constant uniformly in~$j$. By \eqref{eq:elliptic-ker} there is a positive constant $C$, uniformly in $j$, such that 
\begin{align*}
	\left\|\frac{ \hat{h}(t_j)-h} {t_j -a} - w(t_j)   \right\|_{\C^{k,\alpha}(\Omega)}&\le C \left\|f_j \right\|_{\mathcal{B}}.
\end{align*}
Since the right hand side is bounded from above uniformly in $j$, by passing to a subsequence, which we still label by $j$, we have 
\[
	\frac{ \hat{h}(t_j)-h} {t_j -a} - w(t_j) \to p \quad  \mbox{ in } \C^{k,\alpha}(\Omega)
\]
and $\|(t_j -a) w(t_j)\|_{\C^{k,\alpha}(\Omega)}\to 0 $. Finally, let $h(t_j) =  \hat{h}(t_j)-(t_j-a) w(t_j)$, then it satisfies the desired properties. 
\end{proof}

We restate the preceding proposition for the ``ungauged'' operators $L_{t}$, which will be used in the subsequent discussion.


From this point forward, to establish a connection between $L_{t}$ and the gauged operator $L^{\Ga}_{t}$, we revisit the no-kernel assumption for the case $\Lambda>0$ given by \eqref{eq:NK}. We also recall the operator $B$ from \eqref{eq:B}. Let us define the pull-back operators as follows:
\begin{align*}
	B_t Y = \left\{ \begin{array}{ll} \Delta_{g(t)} Y + (n-1)\Lambda Y  & \mbox{ in } \Omega\\
	Y & \mbox{ on } \Sigma \end{array}\right..
\end{align*}
Define $J_0 = \big\{t \in [0, \delta): \Ker B_t = \{ 0 \}\big\}$. In other words, for $t\in J_0$, $(\Omega, g(t))$ satisfies the no-kernel assumption. Lemma~\ref{le:motivation} implies that $J_0$ is open and dense.

\begin{definition}\label{de:J}
For the case $\Lambda \le 0$, we define the open dense subset $J \subset [0,\delta)$ as specified in \eqref{eq:J}. In the case of $\Lambda > 0$, we replace $J$ in \eqref{eq:J} with $J \cap J_0$, which is also an open and dense subset of $[0,\delta)$.
\end{definition}

\begin{corollary}\label{co:kernel}
Let $J$ be the open dense subset defined in Definition~\ref{de:J}. Then, for any $a\in J$ and any $h\in \Ker L_a$, there exists a sequence $\{ t_j \}\subset J$ with $t_j \searrow a$, $h(t_j)\in \Ker L_{t_j}$, and a symmetric $(0,2)$-tensor $p\in \C^{k,\alpha}(\Omega)$ such that as $t_j \searrow a$
\[
	h(t_j) \to h \quad \mbox{ and } \quad \frac{h(t_j)- h}{t_j -a} \to p  \quad \mbox{ in } \C^{k,\alpha}(\Omega).
\]
\end{corollary}
\begin{proof}
For $h\in \Ker L_a$, there is a vector field $W$ with $W|_\Sigma = 0$ such that $\tilde{h}:= h+\mathscr L_W \bar{g}$ satisfies the harmonic gauge, $\b_{g(a)} \tilde h=0$, and thus $L^\Ga_a (\tilde{h})=0$. (For the case $\Lambda>0$, we use the no-kernel assumption~\eqref{eq:NK}.) By Proposition~\ref{pr:kernel}, there exists a sequence $\tilde{h}(t_j) \in \Ker L^{\Ga}_{t_j}$ and a symmetric $(0,2)$-tensor $\tilde{p}\in \C^{k,\alpha}(\Omega)$ such that 
\[
    \tilde{h}(t_j) \to \tilde{h} \quad \text{ and } \quad \frac{\tilde{h}(t_j) - \tilde{h}}{t_j-a } \to \tilde{p} \quad \text{ in } \C^{k,\alpha}(\Omega).
\]
Define $h(t_j):= \tilde{h} (t_j) - \mathscr{L}_W g(t_j)$, and note that it belongs to $\Ker L_{t_j}$ because $\tilde{h} (t_j) \in \Ker L_{t_j}$ by Lemma~\ref{le:gauge-linear}. Then
\begin{align*}
    h(t_j) &\too  \tilde{h} - \mathscr{L}_W \bar{g}  = h,\\
    \frac{h(t_j) - h }{t_j-a} = \frac{\tilde{h}(t_j) - \mathscr{L}_W g(t_j) - (\tilde{h} - \mathscr{L}_W \bar{g}) } {t_j-a} &\too \tilde{p}  -\mathscr{L}_W \mathscr{L}_V \bar{g}
\end{align*}
in $\C^{k,\alpha}(\Omega)$, where we recall  $V$ is the deformation vector of $\{ \Omega_t\}$.  Letting $p= \tilde{p}  - \mathscr{L}_W \mathscr{L}_V \bar{g}$ concludes the proof.
\end{proof}

We prove the main result of this section. 

\begin{theorem}\label{th:generic-bdry}
    Let $\{\Omega_t\}$ be defined as in Notation~\ref{no:foliation}, and let $J\subset [0, \delta)$ be the  open dense subset defined in Definition~\ref{de:J}. For $a\in J$ and $L_a( h)=0$,  the following results hold: 
    \begin{enumerate}
        \item \label{it:1} Assume either $n=3$, or $n>3$ and $\Sigma_a$ is umbilic. Then $h$ satisfies
        \[
            A'|_{g(a)}(h ) = \tfrac{1}{n-1} (\tr_{g(a)} h^\intercal ) A_{g(a)} \quad \text{on } \Sigma.
        \]
        \item \label{it:2} Assume $n\geq 4$, $\Lambda=0$.  Suppose either (i) $\zeta=1$ and $(\Omega, g(a))$ satisfies the non-degenerate boundary condition \eqref{eq:bd1}, or (ii) $\zeta = H_{\bar g}$ and $(\Omega, g(a))$ satisfies the $H$-non-degenerate boundary condition \eqref{eq:bd2} and has strictly mean convex boundary, i.e. $H_{\bar g}>0$. Then $h$ satisfies the Cauchy boundary condition 
        \begin{align*}
            h^\intercal = 0 \quad \text{and} \quad A'|_{g(a)}(h ) = 0 \quad \text{on } \Sigma.
        \end{align*}
    \end{enumerate}
\end{theorem}
\begin{proof}
Let $a\in J$. By relabeling, we may, without loss of generality, assume $a=0$, and denote $ g(a) = \bar g$ and $L_a = L$. Let $h$ satisfy $L(h)=0$. As before, for linearizations and geometric operators taken with respect to $\bar g$, we omit the subscripts.  We further assume that $h$ satisfies the geodesic gauge $h(\nu,\cdot)=0$ on $\Sigma$ (see, e.g., \cite[Lemma 2.5]{An-Huang:2022} for existence of the geodesic gauge).

We denote  $\psi_{t_j}^* h := \psi_{t_j}^* \big(h|_{\Omega_{t_j}}\big) $ on $\Omega$. Let  $h(t_j)$ come from Corollary~\ref{co:kernel}. Observe that $h(t_j) - \psi_{t_j}^* h \to 0$ as $t_j\to 0$. Define $w$ as 
\begin{align*}
    w = \lim_{t_j\to 0^+} \frac{h(t_j) - \psi_{t_j}^* h}{t_j} =\lim_{t_j\to 0^+} \frac{h(t_j)  - h}{t_j} -\lim_{t_j\to 0^+} \frac{ \psi_{t_j}^* h - h}{t_j}  = p - \mathscr{L}_V h.
\end{align*}
We recall $V|_\Sigma = - \zeta \nu$. For the rest of the arguments, we slightly abuse the terminology and use ``differentiating in $t$ and evaluating at $t=0$'' to refer to taking the difference quotient in $t_j$ and letting $t_j\to 0^+$.

We \emph{claim} the following:
\begin{align}
    &\Ric'(w) - (n-1) \Lambda w = 0 \quad \text{in } \Omega\label{eq:int}\\
    &\left\{ \begin{array}{l} w^\intercal - \tfrac{1}{n-1} (\tr w^\intercal) \bar g^\intercal = 2\zeta  \left( A'(h) -\tfrac{1}{n-1} (\tr h^\intercal) A \right)\\
    H'(w) = - 2\zeta \big(A \cdot A'(h) - \tfrac{1}{n-1}\tr h^\intercal  |A|^2 \big)\end{array} \right.\quad   \text{on } \Sigma.\label{eq:bd}
\end{align}

To establish \eqref{eq:int}, we differentiate $\Ric'|_{g(t)}\big(h(t)- \psi^*_t (h) \big) - (n-1) \Lambda\big(h(t)- \psi^*_t (h) \big) = 0$ in $t$ and evaluate at $t=0$. Since all $t$-derivatives that don't involve $h(t)- \psi^*_t (h)$ vanish when evaluated at $t=0$, we derive \eqref{eq:int}.

For the first identity of \eqref{eq:bd}, we differentiate $h(t)^\intercal - \tfrac{1}{n-1} (\tr_{g(t)} h(t)^\intercal ) g(t)^\intercal=0$ in $t$ and evaluating at $t=0$ and obtain 
\[
    p^\intercal + \tfrac{1}{n-1} (g'(0)^\intercal \cdot h^\intercal )\bar g^\intercal - \tfrac{1}{n-1} (\tr p^\intercal )\bar g^\intercal - \tfrac{1}{n-1} (\tr h^\intercal) g'(0)^\intercal = 0.
\]
Rearranging terms and using $h^\intercal = \tfrac{1}{n-1}( \tr h^\intercal )\bar g^\intercal$ and $g'(0)^\intercal=( \mathscr L_V \bar g)^\intercal=-2\zeta A$, we obtain 
\begin{align}\label{eq:p-bdry}
    \begin{split}
        p^\intercal - \tfrac{1}{n-1} (\tr p^\intercal) \bar g^\intercal &= \tfrac{1}{n-1} \tr h^\intercal \left( g'(0)^\intercal - \tfrac{1}{n-1} (\tr g'(0)^\intercal) \bar g^\intercal \right)\\
        &=-\tfrac{2}{n-1}\zeta \tr h^\intercal \left(A -\tfrac{1}{n-1} H \bar g^\intercal \right).
    \end{split}
\end{align}
Since $h$ is in the geodesic gauge, $(\mathscr L_Vh)^\intercal = -2\zeta A'(h)$ by \eqref{eq:A}. Therefore,
\begin{align}\label{eq:LX}
    \begin{split}
        (\mathscr L_Vh)^\intercal  - \tfrac{1}{n-1} (\tr (\mathscr L_Vh)^\intercal ) \bar g^\intercal &= -2\zeta \left(A'(h) - \tfrac{1}{n-1} (\tr A'(h)) \bar g^\intercal\right)\\
        &=- 2\zeta \left(A'(h) - \tfrac{1}{(n-1)^2} (\tr h^\intercal )H \bar g^\intercal\right)
    \end{split}
\end{align}
where we use $\tr A'(h) = H'(h) + h^\intercal \cdot A = \tfrac{1}{n-1} (\tr h^\intercal ) H$ from the boundary conditions of $h$. Subtracting \eqref{eq:p-bdry} from \eqref{eq:LX} gives the first identity in \eqref{eq:bd}.

We compute the second identity in \eqref{eq:bd}:
\begin{align*}
    H'(w) &= H'(p-\mathscr L_V h) = \left. \dt\right|_{t=0} H'|_{g(t)}(h(t) - \psi_t^* h)\\
    &= \left. \dt\right|_{t=0} H'|_{g(t)}( - \psi_t^* h)\\
    &= - \left. \dt\right|_{t=0}\psi_t^* (H'(h))\\
    &= \zeta \nu(H'(h))\\
    &= -2\zeta ( A\cdot A'(h)- A\cdot (A\circ h^\intercal ) )\\
    &= - 2\zeta \left(A\cdot A'(h)-  \tfrac{1}{n-1}\tr h^\intercal |A|^2\right)
\end{align*}
where we use $H'|_{g(t)}(h(t))=0$ in the second line, $h^\intercal=\tfrac{1}{n-1} \tr h^\intercal \bar g^\intercal$ in the last line, and the linearized Ricatti equation in the fifth line:
\[
    \nu(H'(h)) = (\nu(H) )'(h) = -(|A|^2 + \Ric(\nu, \nu))' (h)= -2 ( A\cdot A'(h)- A\cdot (A\circ h^\intercal ) )
\] 
where $(A_g\circ h^\intercal )_{ab} = \tfrac{1}{2} (A_{ac} h^c_{b}+A_{bc} h^c_{a})$ and we use $\nu'(h)=0, \Ric'(\nu,\nu)=(n-1)\Lambda h(\nu,\nu)=0$ since $h$  is in the geodesic gauge. It completes the proof of \eqref{eq:bd}. 

We insert $h, w$ into the Green-type identity \eqref{eq:Green-cor} and use $L(h)=0$, \eqref{eq:int}, and \eqref{eq:bd} to compute
\begin{align}\label{eq:Gb}
 & \int_\Sigma\left\langle \big( A'(h) + \big(\tfrac{1}{2} - \tfrac{1}{n-1} \big) \tr h^\intercal A, \big( 1-\tfrac{1}{n-1} \big)\tr h^\intercal\big),  \big(w^\intercal - \tfrac{1}{n-1} \tr w^\intercal \bar g^\intercal, H'(w)\big) \right\rangle  \da\notag\\
 \begin{split}
 &=\int_\Sigma \big( A'(h) -\tfrac{1}{2}  \tr h^\intercal A\big)\cdot  (w^\intercal - \tfrac{1}{n-1} \tr w^\intercal \bar g^\intercal) \da\\
 &=2 \int_\Sigma  \zeta \big( A'(h) -\tfrac{1}{2}  \tr h^\intercal A\big)\cdot \big( A'(h) -\tfrac{1}{n-1} (\tr h^\intercal) A \big)\da\\
 &= 2 \int_\Sigma \zeta |A'(h) - \tfrac{1}{n-1} (\tr h^\intercal) A|^2 \da \\
 &\quad - 2\left(\tfrac{1}{2} - \tfrac{1}{n-1} \right) \int_\Sigma \zeta \tr h^\intercal A\cdot \big( A'(h) -\tfrac{1}{n-1} (\tr h^\intercal) A \big) \da = 0.
 \end{split}
\end{align}
Thus, we obtain
\begin{align} \label{eq:zero}
\begin{split}
 & \int_\Sigma \zeta |A'(h) - \tfrac{1}{n-1} (\tr h^\intercal) A|^2 \da \\
 &= \left(\tfrac{1}{2} - \tfrac{1}{n-1} \right) \int_\Sigma \zeta \tr h^\intercal A\cdot \big( A'(h) -\tfrac{1}{n-1} (\tr h^\intercal) A \big) \da.
 \end{split}
\end{align}

If $n=3$, the right-hand side of \eqref{eq:zero} is zero. If $n>3$ and $\Sigma$ is umbilic, then the right-hand side is also zero because
\[
A\cdot \big( A'(h) -\tfrac{1}{n-1} (\tr h^\intercal) A \big)= \tfrac{1}{n-1} H \tr \big( A'(h) -\tfrac{1}{n-1} (\tr h^\intercal) A \big) = 0.
\]
In either case, we obtain $A'(h) - \tfrac{1}{n-1} (\tr h^\intercal) A=0$, which proves Item~\eqref{it:1}. 

Assume the hypothesis of Item~\eqref{it:2}. Observe that by \eqref{eq:hiddenbd2},  the integral in the right-hand side of \eqref{eq:zero} can be expressed as 
\begin{align*}
    \int_\Sigma \zeta \tr h^\intercal A\cdot \big( A'(h) -\tfrac{1}{n-1} (\tr h^\intercal) A \big) \da=  -\tfrac{(n-2)}{2(n-1)}\int_{\Sigma_t}\zeta  \tr h^\intercal L_\Sigma (\tr h^\intercal) \da.
\end{align*}
Since $\int_\Sigma ( \tr h^\intercal ) H\da = 0$ by  \eqref{eq:trH}, we can invoke either  the non-degenerate boundary condition \eqref{eq:bd1} or  the $H$-nondegenerate boundary condition \eqref{eq:bd2} (by letting $v=  \tr h^\intercal$ in either one) to conclude that the right-hand side is strictly negative unless $\tr h^\intercal =0$. But \eqref{eq:zero} implies the right-hand side is nonnegative. We conclude $\tr h^\intercal =0$ and $A'(h) - \tfrac{1}{n-1} (\tr h^\intercal) A= A'(h)=0$. It completes the proof.

\end{proof}

Combining the above theorem with Proposition~\ref{pr:zero} and Theorem~\ref{th:infinitesimal} gives the following corollary.

\begin{corollary}\label{co:rigid}
    Let $\{\Omega_t\}$ be defined as in Notation~\ref{no:foliation}, and let $J$ be the open dense subset defined in Definition~\ref{de:J}. For $a\in J$ and $L_a( h)=0$,  the following results hold: 
  \begin{enumerate}
\item \label{it:rigid} (Infinitesimal rigidity) Suppose either one of the following holds:
\begin{enumerate}
\item  $n=3$, $\Sigma_a$ is not umbilic, and $R_{\Sigma_a}$ is not identically zero
\item $n\ge 4$, $\Lambda=0$, and suppose either (i) $\zeta=1$, $(\Omega, g(a))$ satisfies the non-degenerate boundary condition \eqref{eq:bd1} or (ii)  $\zeta = H_{\bar g}$ and $(\Omega, g(a))$ satisfies the $H$-non-degenerate boundary condition \eqref{eq:bd2} and has strictly mean convex boundary, i.e. $H_{\bar g}>0$.
\end{enumerate}
Then $h =\mathscr L_X g(a)$ in $\Omega$ for some $X\in \mathcal X^{k+1,\alpha}(\Omega)$. That is, $(\Omega,  g(a))$ is  infinitesimally rigid  with respect to the Anderson boundary data. (Recall Definition~\ref{de:rigid}.)

\item \label{it:b}  (Infinitesimal non-rigidity) If $n\ge 3$ and the boundary $(\Sigma_a, g^\intercal(a))$ is a round sphere and umbilic, then $h =\mathscr L_X g(a)$ for some $X\in \mathcal Z^{k+1,\alpha} (\Omega)$. 
\end{enumerate}  
\end{corollary}

\subsection{Proofs of Theorems~\ref{th:generic0}, \ref{th:generic-higher}, and \ref{th:bdrymap}}
We prove Theorem~\ref{th:generic0} (for $n=3$) and Theorem~\ref{th:generic-higher} (for $n\geq 4$) together. The distinction in the proof arises from the application of Corollary~\ref{co:rigid} in these two cases.

\begin{proof}[Proof of Theorem~\ref{th:generic0} and Theorem~\ref{th:generic-higher}]

Let $\bar g$ represent $\llbracket\bar g\rrbracket \in \mathcal M_\Lambda^{k,\alpha}(\Omega)/\mathscr D^{k+1,\alpha}(\Omega)$. Consider the inward deformation $\Omega_t \subset (\Omega, \bar g)$ for $t\in [0, \delta)$ as defined in Notation~\ref{no:foliation}, and note that $g(t)$ converges to $\bar g$ in $\C^{k,\alpha}$.

For the case when $n=3$, we can deform $\Omega$ such that the resulting boundary $\Sigma_t$ is not umbilic, and $R_{\Sigma_t} \not\equiv 0$ for all $t\in (0, \delta)$. In the case when $n\ge 4$, we let $\bar g \in \widehat {\mathcal M}^{k,\alpha}_0 (\Omega)$ and deform $\Omega$ by either $\zeta=1$ or $\zeta = H_{\bar g}$. 

According to Item~\eqref{it:rigid} in Corollary~\ref{co:rigid}, there exists an open dense subset $J\subset [0, \delta)$ such that $(\Omega, g(t))$ is infinitesimally rigid with respect to the Anderson boundary data for all $t\in J$. By Proposition~\ref{pr:Fredholm}, $\Pi$ is a local diffeomorphism at $\llbracket g(t) \rrbracket$ for all $t\in J$.
\end{proof}

We explore the dilation invariance of star-shaped regions $\Omega\subset \mathbb R^n$ and apply Corollary~\ref{co:rigid} to obtain the following corollary.
\begin{corollary}\label{co:star}
Let $\Omega$ be a star-shaped region in the Euclidean space $(\mathbb R^n, \bar g)$. Then the following holds:
\begin{enumerate}
\item \label{it:3} If $n=3$ and $\Omega$ is not a round ball, then $(\Omega, \bar g)$ is  infinitesimally rigid  with respect to the Anderson boundary condition. 
\item If $n\ge 3$ and $\Omega$ is a round ball, then for any $h$ solving $L(h)=0$, we have $h= \mathscr L_X \bar g$ for some $X \in \mathcal Z^{k+1,\alpha}(\Omega)$.  \label{it:sphere}
\end{enumerate} 
\end{corollary}
\begin{proof}

We assume $\Omega$ is star-shaped with respect to the origin. 
Let $t\in [0, 1)$ and $\psi_t: \mathbb R^n\to \mathbb R^n$ be the dilation map defined by $\psi_t(x_1, \dots, x_n) = (1-t)(x_1, \dots, x_n)$. Restricting $\psi_t: \Omega\to \Omega_t$, we obtain a smooth family of inward perturbed domains $\Omega_t$ as defined in Notation~\ref{no:foliation}. Let $g(t) = \psi_t^* (\bar g|_{\Omega_t})$, and we calculate $g(t) = (1-t)^2\bar g$. 

Observe that any kernel element $h \in \Ker L_0$ is also a kernel element of the pull-back operator \( L_t \) defined by \eqref{eq:Lt} for all $t$: Using the linearized formulas \eqref{eq:H} and \eqref{equation:Ricci}, and the fact that \( g(t) =  (1-t)^{2} \bar{g} \), we can compute 
\begin{align*}
    \Ric'\big|_{g(t)}(h) & = (1-t)^{-2} \Ric'\big|_{\bar{g}}(h), \\
    h^\intercal - \tfrac{1}{n-1} (\mathrm{tr}_{g(t)} h^\intercal)g(t)^\intercal & = h^\intercal - \tfrac{1}{n-1} (\mathrm{tr}_{\bar{g}} h^\intercal)\bar{g}^\intercal, \\
    H'\big|_{g(t)}(h) & = (1-t)^{-3} H'\big|_{\bar{g}}(h).
\end{align*}
Since we know $\Ker L_t$ for generic $t$ from Corollary~\ref{co:rigid}, we obtain the desired result.
 
\end{proof}

\begin{proof}[Proof of Theorem~\ref{th:bdrymap}]
Let $\Omega$ be a star-shaped region that is not a round ball in the Euclidean space $(\mathbb R^3, \bar g)$. By Corollary~\ref{co:star},  $(\Omega, \bar g)$ is  infinitesimally rigid  with respect to the Anderson boundary condition. By Proposition~\ref{pr:Fredholm}, the boundary map $\Pi: \mathcal M^{k,\alpha}_0(\Omega) / \mathscr D^{k+1,\alpha}(\Omega) \to \mathcal S_1^{k,\alpha}(\Sigma)\times \C^{k-1,\alpha}(\Sigma)$ is a local diffeomorphism at $\llbracket \bar g \rrbracket$. The desired result follows the identification of the moduli space of the Ricci flat metric with the moduli space of immersions,  Lemma~\ref{le:embedding}.
  
\end{proof}

\section{Negative Einstein metrics in higher dimensions}\label{se:nonzero}

In this section,  we prove Theorem~\ref{th:general}. Throughout the section, we assume  $n \geq 4$ and $\Lambda < 0$. When $\Lambda \neq 0$, we no longer have the zero mean value property Lemma~\ref{le:trH}, which is used in conjunction with the non-degenerate boundary conditions. Instead, the mean value is determined by the deformation of the interior volume, which motivates us to consider the modified map below. We also note that while most of the arguments in this section hold for any value of $\Lambda$, we use $\Lambda<0$ in \eqref{eq:bulk-sign}.

Fix a background metric $\bar g$ such that $\Ric_{\bar g} = (n-1)\Lambda \bar g$. Denote by $V_g = \mathrm{vol}(\Omega, g)$. Define the map $\overline T:\mathcal M^{k,\alpha}(\Omega)\times \mathbb R\to \C^{k-2,\alpha}(\Omega)\times \mathcal S_1^{k,\alpha}(\Sigma)\times \C^{k-1,\alpha}(\Sigma) \times \mathbb R$ by 
\begin{align*}
	\overline T(g, \eta) =  \left\{\begin{array}{ll} \Ric_g - (n-1) \Lambda \eta g \quad \mbox{ in } \Omega \\
	\left\{ \begin{array}{ll} |g|^{-\frac{1}{n-1}}g^\intercal \\
	H_g \end{array} \right. \mbox{ on } \Sigma\\
		2\eta^{\frac{n}{2}} V_g 
	\end{array}  \right..
\end{align*}
The linearization at $(\bar g, 1)$ is given by $\overline L:  \C^{k,\alpha}(\Omega) \times \mathbb R\to \C^{k-2,\alpha}(\Omega)\times T\mathcal S_1^{k,\alpha}(\Sigma)\times \C^{k-1,\alpha}(\Sigma)\times \mathbb R$ where 
\begin{align*}
	\overline L(h, b) = \left\{\begin{array}{ll} \Ric'(h) - (n-1) \Lambda h - (n-1)\Lambda b  \bar g \quad \mbox{ in } \Omega \\
	\left\{ \begin{array}{ll} h^\intercal- \tfrac{1}{n-1} (\tr h^\intercal) \bar g^\intercal \\
	H'(h) \end{array} \right. \mbox{ on } \Sigma\\
		\int_\Omega \tr h \dvol + n V_{\bar g} b
	\end{array}  \right..
\end{align*}
As in the previous sections, we omit the subscript $\bar g$ when referring to linearizations and geometric operators taken at the background $\bar g$. 

Define the corresponding gauged operator $\overline T^\Ga:\mathcal M^{k,\alpha}(\Omega)\times \mathbb R\to \C^{k-2,\alpha}(\Omega)\times \C^{k-1,\alpha}(\Sigma)\times \mathcal S_1^{k,\alpha}(\Sigma)\times \C^{k-1,\alpha}(\Sigma)\times \mathbb R$ by 
\begin{align*}
	\overline T^\Ga(g, \eta) =  \left\{\begin{array}{ll} \Ric_g - (n-1) \Lambda \eta g  + \mathcal D_g \beta_{\bar g} g \quad \mbox{ in } \Omega \\
	\left\{ \begin{array}{ll} \beta_{\bar g} g\\ |g|^{-\frac{1}{n-1}}g^\intercal \\
	H_g \end{array} \right. \mbox{ on } \Sigma\\
		2\eta^{\frac{n}{2}} V_g 
	\end{array}  \right..
\end{align*}
The linearization at $(\bar g, 1)$ is given by $\overline L^\Ga:  \C^{k,\alpha}(\Omega) \times \mathbb R\to \C^{k-2,\alpha}(\Omega)\times \C^{k-1,\alpha}(\Sigma)\times T\mathcal S_1^{k,\alpha}(\Sigma) \times  \C^{k-1,\alpha}(\Sigma)\times \mathbb R$ where 
\begin{align*}
	\overline L^\Ga(h, b) = \left\{\begin{array}{ll} \Ric'(h) - (n-1) \Lambda h - (n-1)\Lambda b  \bar g +  \mathcal D \beta h\quad \mbox{ in } \Omega \\
	\left\{ \begin{array}{ll} \beta h\\ 
	h^\intercal- \tfrac{1}{n-1} (\tr h^\intercal) \bar g^\intercal \\
	H'(h) \end{array} \right. \mbox{ on } \Sigma\\
		\int_\Omega \tr h \dvol + n V_{\bar g} b
	\end{array}  \right..
\end{align*}
We refer the definitions of above function spaces to the definitions of $T, L, T^\Ga, L^\Ga$ in \eqref{eq:T}, \eqref{eq:L}, \eqref{eq:TG}, \eqref{eq:LG} respectively.

Using the modified operator, we recover the zero mean value condition.  
\begin{lemma}[Zero mean value for $\Lambda \neq 0$]\label{le:vanish}
Suppose $\overline L(h, b)=0$. Then 
\[
\int_\Sigma (\tr h^\intercal) H\da=0.
\] 
\end{lemma}
\begin{proof}
We follow closely Lemma~\ref{le:trH}.  Let $w=\bar g$ in the Green-type identity \eqref{eq:Green-cor} and we have 	$P(\bar g)= -\tfrac{1}{2} (n-1)(n-2) \Lambda \bar g$. We get $P(h) = \tfrac{1}{2}  (n-1)(n-2)\Lambda b \bar g$ by \eqref{eq:K}. Therefore, 
\begin{align*}
	&\int_\Omega P(h) \cdot \bar g \dvol- \int_\Omega P(\bar g)\cdot h  \dvol \\
	&= \tfrac{1}{2} (n-1)(n-2) n\Lambda b V_{\bar g} +\tfrac{1}{2} (n-1)(n-2) \Lambda \int_\Omega \tr h \dvol= 0.
\end{align*}	
Since the boundary terms of the Green identity are the same as in Lemma~\ref{le:trH}, we complete the proof.
\end{proof}

We establish some basic properties of the modified operator. The next lemma follows similarly to Lemma~\ref{le:gauge-linear}.
\begin{lemma}\label{le:overlineL}
\begin{enumerate}
\item If $h\in \Ker \overline L^\Ga $, then $\beta_{\bar g} h =0$ and $h\in \Ker \overline L$.
\item If $h\in \Ker \overline L$, then there exists a unique $Y\in \mathcal X^{k+1,\alpha}(\Omega)$ such that $\beta_{\bar g} (h+ \mathscr L_Y \bar g)=0$ and  $h+ \mathscr L_Y\bar g\in \Ker L^\Ga$. 
\end{enumerate}
\end{lemma}

\begin{proposition}\label{pr:sur}
The operator $\overline L^\Ga$ is Fredholm of index $0$ and satisfies the Schauder estimate
\[
	\|(h, b) \|_{\C^{k,\alpha}(\Omega)\times \mathbb R} \le C \big(\| \overline L^\Ga(h, b) \|_{\mathcal B} + \| (h, b)\|_{\C^{0}(\Omega)\times \mathbb R} \big)
\]	
where $\mathcal B$ denotes the norm of the codomain space of $ \overline L^\Ga$ and the constant $C$ is uniform in the background metric $\| \bar g\|_{\C^{k,\alpha}(\Omega)}$.
\end{proposition}
\begin{proof}
To verify the Schauder estimate, notice that the first four components of $\overline L^\Ga(h, b)$ form exactly $L^\Ga(h)$ with the inhomogeneous term $(n-1) \Lambda b \bar g$. Since $L^\Ga$ is elliptic, we have 
\[
		\|h \|_{\C^{k,\alpha}(\Omega)} \le C \big(\| \overline L^\Ga(h, b) \|_{\mathcal B} + \| (n-1) \Lambda b \bar g \|_{\C^{k-2,\alpha}(\Omega)}+ \| h\|_{\C^{0}(\Omega)} \big).
\]
It implies the desired estimate.

The Schauder estimate implies that $\overline L^\Ga$ has a  finite-dimensional kernel and closed range. We will show that $\Dim \Ker \overline L^\Ga = \Dim \Coker \overline L^\Ga$.

Define the operator 
\begin{align*}
	\widehat L^\Ga(h) = \left\{\begin{array}{ll} \Ric'(h) - (n-1) \Lambda h + \frac{n-1}{n V_{\bar g}}\Lambda\left(\int_\Omega \tr h \dvol \right)   \bar g +  \mathcal D \beta h\quad \mbox{ in } \Omega \\
	\left\{ \begin{array}{ll} \beta h\\ 
	h^\intercal- \tfrac{1}{n-1} (\tr h^\intercal) \bar g^\intercal \\
	H'(h) \end{array} \right. \mbox{ on } \Sigma
	\end{array}  \right..
\end{align*}
Since $\widehat L^\Ga$ is just the operator $L^\Ga$ with  an extra inconsequential zeroth order ``non-local'' term, one can verify that $\widehat L^\Ga$ is also of Fredholm index $0$, just as $L^\Ga$.  

Our motivation for defining $\widehat L^\Ga$ is that 
\[
	\Ker \overline{L}^\Ga = \Big\{ (h, b):  h\in \Ker \widehat L^\Ga \mbox{ and } b = -\tfrac{1}{n V_{\bar g} }\int_\Omega \tr h \dvol \Big\}.
\]
Thus, the kernel spaces for $\overline L^\Ga$ and for $ \widehat L^\Ga$ have the same dimensions.  We will show that their cokernel spaces are isomorphic. Then we can conclude that the Fredholm index of $\overline{L}^\Ga$ is also zero. 

We denote $(I, B, c)$ as the elements in the dual space of the codomain of $\overline L^\Ga$. For any $(h, b)$, the pairing with $(I, B, c)$ satisfies 
\begin{align}\label{eq:pairing}
\begin{split}
	&\Big\langle \overline L^\Ga(h, b), (I, B, c) \Big\rangle \\
	&=\Big \langle \widehat L^\Ga(h), (I, B)\Big \rangle + \left(c - \tfrac{n-1}{nV_{\bar g}}\Lambda \int_\Omega \tr  I  \dvol \right) \left( \int_\Omega \tr h \dvol + n V_{\bar g} b\right).
\end{split}
\end{align} 
To see it, since the pairing of the boundary terms are obviously the same, we just need to compute 
\begin{align*}
	&\int_\Omega \left\langle \Ric'(h) - (n-1) \Lambda h - (n-1)\Lambda b  \bar g +  \mathcal D \beta h, I\right\rangle \dvol  + c\Big(\int_\Omega \tr h \dvol + n V_{\bar g} b\Big)\\
	&=\int_\Omega \Big\langle \Ric'(h) - (n-1) \Lambda h + \tfrac{n-1}{n V_{\bar g}}\Lambda\Big(\int_\Omega \tr h \dvol \Big)   \bar g +  \mathcal D \beta h, I \Big\rangle \dvol\\
	&\quad +\left(- (n-1) \Lambda b - \tfrac{n-1}{n V_{\bar g}} \Lambda \int_\Omega \tr h \dvol \right) \int_\Omega \tr I \dvol+c\Big(\int_\Omega \tr h \dvol + n V_{\bar g} b\Big).
\end{align*}
Grouping the common factor $(\int_\Omega \tr h \dvol + n V_{\bar g} b)$  gives \eqref{eq:pairing}. 

Let $(I, B, c) \in \Coker \overline L^\Ga$; that is, $\Big\langle \overline L^\Ga(h, b), (I, B, c) \Big\rangle=0$ for all $(h, b)$. Then for any given $h$ if we choose $b$ so that $ \int_\Omega \tr h \dvol + n V_{\bar g} b=0$, then $\Big \langle \widehat L^\Ga(h), (I, B)\Big \rangle=0$ by \eqref{eq:pairing}, and thus $(I, B)\in \Coker \widehat L^\Ga$. Conversely, if $(I, B)\in \Coker  \widehat L^\Ga$, then it is easy to check that $(I, B, c)\in \Coker \overline L^\Ga$ where $c=\tfrac{n-1}{nV_{\bar g}}\Lambda \int_\Omega \tr  I  \dvol$. Thus, the two cokernel spaces are isomorphic. It completes the proof.

\end{proof}

Following Notation~\ref{no:foliation}, we let $\Omega_t$ be a family of inward normal deformations $\psi_t : \Omega \to \Omega_t$ where $\frac{\partial}{\partial t} \psi_t = V$ and $V|_\Sigma = -\zeta \nu$ where $\nu$ is the outward unit normal. 
Denote the pull-back metrics $g(t) = \psi^*_t (\bar g|_{\Omega_t})$. Let $\overline L_t$ be the linearized operator of $\overline T$ at $(g(t), 1)$; namely,
\begin{align*}
	\overline L_t(h, b) = \left\{\begin{array}{ll} \Ric'|_{g(t)}(h) - (n-1) \Lambda h - (n-1)\Lambda b   g(t) \quad \mbox{ in } \Omega \\
	\left\{ \begin{array}{ll} h^\intercal- \tfrac{1}{n-1} (\tr_{g(t)} h^\intercal) \bar g(t)^\intercal \\
	H'|_{g(t)}(h) \end{array} \right. \mbox{ on } \Sigma\\
		\int_\Omega \tr_{g(t)} h \dvol_{g(t)} +n V_{g(t)} b
	\end{array}  \right..
\end{align*}
The corresponding gauged operator is given by
\begin{align*}
	\overline L_t^\Ga(h, b) = \left\{\begin{array}{ll} \Ric'|_{g(t)}(h) - (n-1) \Lambda h - (n-1)\Lambda b   g(t) + \mathcal D_{g(t)} \beta_{g(t)} h  \quad \mbox{ in } \Omega \\
	\left\{ \begin{array}{ll} \beta_{g(t)} h\\
	h^\intercal- \tfrac{1}{n-1} (\tr_{g(t)} h^\intercal) \bar g(t)^\intercal \\
	H'|_{g(t)}(h) \end{array} \right. \mbox{ on } \Sigma\\
		\int_\Omega \tr_{g(t)} h \dvol_{g(t)} +n V_{g(t)} b
	\end{array}  \right..
\end{align*}

 Let $N_1(t) = \dim\overline L_t^\Ga$. 
\begin{align*}
	J_1 = \{ t\in [0, \delta): \mbox{$N_1(t)$ is constant in an open neighborhood of $t$} \}.
\end{align*}
Since $\overline L^\Ga_t$ satisfies a uniform Schauder estimate, the analogous Lemma~\ref{le:nullity}, as well as its consequential Proposition~\ref{pr:kernel} and Corollary~\ref{co:kernel},  also hold for the modified operators $\overline L^\Ga_t$ and $N_1(t), J_1$. We summarize the results below. 
\begin{proposition}\label{pr:conv}
For any $a\in J_1$ and any $(h, b)\in \Ker \overline L_a$, there exists a sequence $\{ t_j \}\subset J$ with $t_j \searrow a$, $(h(t_j), b(t_j))\in \Ker \overline L_{t_j}$, and a symmetric $(0,2)$-tensor $h'\in \C^{k,\alpha}(\Omega)$ and $b'\in \mathbb R$ such that as $t_j \searrow a$
\[
	(h(t_j), b(t_j)) \to (h, b) \quad \mbox{ and } \quad \frac{(h(t_j)- h, b(t_j)- b)}{t_j -a} \to (h', b')  \quad \mbox{ in } \C^{k,\alpha}(\Omega)\times \mathbb R.
\]
\end{proposition}

We prove the main result of this section that the operators $\overline L_t$ has ``trivial'' kernel for generic $t$. Recall that $\widehat {\mathcal M}^{k,\alpha}_- (\Omega)$ denotes the space of negative Einstein metrics that have $H$-non-degenerate, strictly mean convex boundary.
\begin{theorem}\label{th:trivial}
Let $n\ge 4$ and  $\bar g\in \widehat {\mathcal M}^{k,\alpha}_- (\Omega)$ such that $\Ric  = (n-1) \Lambda \bar g$ for some $\Lambda < 0$. Let the inward domain deformation $\psi_t$ be generated from $V= - H_{\bar g} \nu$, where $H_{\bar g}$ denotes the mean curvature of $\Sigma_t \subset (\Omega, \bar g)$. Then $\overline T^\Ga$ is a local diffeomorphism at $(g(t), 1)$ for $t\in J_1$, where recall $g(t)= \psi_t^*(\bar g|_{\Omega_t})$. 
\end{theorem}
\begin{proof}
It suffices to show that for $a \in J$ and for $(h, a)\in \Ker\overline L_a$,  $h$ satisfies the Cauchy boundary condition $h^\intercal = 0, A'|_{g(a)}(h)=0$ on $\Sigma$ and $b=0$. Once it is obtained, by Proposition~\ref{pr:zero}, $h=\mathscr L_X g(a)$ for some $X\in \mathcal X^{k+1,\alpha}(\Omega)$. By Lemma~\ref{le:overlineL}, $\Ker \overline L^\Ga = \{ 0 \}$. Therefore,  Proposition~\ref{pr:sur} yields that $\overline L^\Ga$ is an isomorphism, and thus $\overline T^\Ga$ is a local diffeomorphism at $g(a)$ by the Inverse Function Theorem.

Let $a\in J_1$ and let $t_j\searrow  a$. Without loss of generality, we may assume $a=0$. By Proposition~\ref{pr:conv}, there is a sequence $(h(t_j), b(t_j))\in \Ker \overline L_{t_j}$ so that  $(h(t_j), b(t_j))\to (h, b)$ and the difference quotient $ \frac{(h(t_j)- h, b(t_j)- b)}{t_j }  \to (h', b')$.  We slightly abuse the terminology and use ``differentiating in $t$ and evaluating at $t=0$'' to refer to taking the difference quotient in $t_j$ and letting $t_j\searrow 0$.

 We have
\begin{align*}
	 \Ric'|_{g(t)}(h(t_j)) - (n-1) \Lambda h(t_j) - (n-1)\Lambda b(t_j)   g(t_j)&=0\\
	 \Ric'|_{g(t_j)} (\psi_{t_j}^* h) - (n-1) \Lambda \psi_{t_j}^* h - (n-1) \Lambda b g(t_j)&=0
\end{align*}
where in the first identity we use that $(h(t_j), b(t_j))\in \Ker \overline L_{t_j}$ and in the second identity we pull back the identity $\Ric'(h) - (n-1)\Lambda h - (n-1) \Lambda b\bar g=0$ restricted on $\Omega_{t_j}$ via $\psi_{t_j}$.  Subtracting the previous two identities, differentiating it in $t$ and evaluating at $t=0$, we obtain
\begin{align*}
	\Ric'(w) - (n-1) \Lambda w - (n-1)\Lambda b' \bar g=0
\end{align*}
where $w:=  h' - \mathscr L_V h$.

For the rest of the proof, we assume that $h$ satisfies the geodesic gauge. We claim the boundary data for $w$:
\begin{align*}
&\left\{ \begin{array}{l} w^\intercal - \tfrac{1}{n-1} (\tr w^\intercal) \bar g^\intercal = 2\zeta  \left( A'(h) -\tfrac{1}{n-1} (\tr h^\intercal) A \right)\\
 H'(w) =  -2\zeta \big(A \cdot A'(h) - \tfrac{1}{n-1}\tr h^\intercal  |A|^2 \big) - \zeta (n-1) \Lambda b\end{array} \right.\quad   \mbox{ on } \Sigma.
\end{align*}
We refer to the same computations as those in \eqref{eq:bd} but we simply highlight the additional term $-\zeta(n-1)\Lambda b$ in the second identity coming from $\Ric'(h)(\nu,\nu)$:
\begin{align*}
    H'(w) &= H'(h'-\mathscr L_V h) \\
    &= \zeta \nu(H'(h))\\
    &= -2\zeta ( A\cdot A'(h)- A\cdot (A\circ h^\intercal ) )  -\zeta(n-1)\Lambda b\\
    &= - 2\zeta \left(A\cdot A'(h)-  \tfrac{1}{n-1}\tr h^\intercal |A|^2\right)-\zeta(n-1)\Lambda b
\end{align*}
where we use
\begin{align*}
    \nu(H'(h)) &= (\nu(H) )'(h) = -(|A|^2 + \Ric(\nu, \nu))' (h)\\
    &= -2 ( A\cdot A'(h)- A\cdot (A\circ h^\intercal ) )  - (n-1) \Lambda b.
\end{align*}

Next, the last component of $\overline L_t(h(t), b(t))=0$ says
\[
- n V_{g(t)} b(t)=\int_\Omega \tr_{g(t)} h(t) \dvol_{g(t)}.
\]
We differentiate it in $t$ and evaluate at $t=0$ to give
\begin{align*}
	-nV_{\bar g} b' - n b \int_\Omega \Div V \dvol  &= \int_\Omega\Big( -\mathscr L_V \bar g \cdot h + \tr h' + \tr h \Div V \Big) \dvol\\
	&= \int_\Omega \tr\Big( h'- \mathscr L_V h \Big) \dvol + \int_\Sigma (\tr h) (V\cdot \nu) \da\\
	&= \int_\Omega \tr w \dvol  + \int_\Sigma (\tr h) (V\cdot \nu) \da
\end{align*}	
where we use the integration by part and  the identity $\tr (\mathscr L_V h) =  V(\tr h) + \mathscr  L_V \bar g \cdot h$. Because $\tr h = \tr h^\intercal$ on $\Sigma$ in the geodesic gauge and $V|_\Sigma= -\zeta \nu$ where $\zeta=H_{\bar g}$, we obtain
\begin{align*}
 \int_\Omega \tr w \dvol =- n V_{\bar g} b' + n b \int_\Sigma \zeta \da + \int_\Sigma \zeta \tr h^\intercal \da= - n V_{\bar g} b' + n b \int_\Sigma \zeta \da.
\end{align*}
Here and below, we use $\int_\Sigma  \zeta \tr h^\intercal\da=\int_\Sigma  (\tr h^\intercal)H_{\bar g} \da=0$ by Lemma~\ref{le:vanish}.

We apply the Green-type identity for $h$ and $w$. By \eqref{eq:K}, we compute 
\[
	P(h) = \tfrac{1}{2} (n-1) (n-2) \Lambda  b \bar g \quad \mbox{ and }\quad  P(w) =  \tfrac{1}{2} (n-1) (n-2) \Lambda  b' \bar g.
\]
We substitute them into the bulk integrals of the Green-type identity \eqref{eq:Green-cor} and get 
\begin{align*}
	&\int_\Omega \langle P(h),  w \rangle \dvol - \int_\Omega \langle P(w),  h \rangle \dvol \\
	&=\tfrac{1}{2} (n-1)(n-2) \Lambda b \int \tr w \dvol -\tfrac{1}{2} (n-1)(n-2) \Lambda b' \int_\Omega \tr h\dvol \\
	&= \tfrac{1}{2} (n-1)(n-2) \Lambda b \left(- n V_{\bar g} b' + nb\int_\Sigma \zeta \da\right)  +\tfrac{1}{2} (n-1)(n-2) \Lambda b'  n V_{\bar g} b \\
	&=\tfrac{1}{2} (n-1)(n-2) n \Lambda b^2 \int_\Sigma \zeta \da.
\end{align*}
We compute the boundary terms of the Green-type identity as in \eqref{eq:Gb} and obtain  
\begin{align*}
 & \int_\Sigma\left\langle \big( A'(h) + \big(\tfrac{1}{2} - \tfrac{1}{n-1} \big) \tr h^\intercal A, \big( 1-\tfrac{1}{n-1} \big)\tr h^\intercal\big),  \big(w^\intercal - \tfrac{1}{n-1} \tr w^\intercal \bar g^\intercal, H'(w)\big) \right\rangle  \da\notag\\
	  &=\int_\Sigma \zeta  |A'(h) - \tfrac{1}{n-1} (\tr h^\intercal) A|^2 \da\\
	  &\quad  - \left(\tfrac{1}{2} - \tfrac{1}{n-1} \right) \int_\Sigma \zeta \tr h^\intercal A\cdot \big( A'(h) -\tfrac{1}{n-1} (\tr h^\intercal) A \big) \da- (n-2) \Lambda b\int_\Sigma \zeta \tr h^\intercal  \da\\
	  &= \int_\Sigma \zeta  |A'(h) - \tfrac{1}{n-1} (\tr h^\intercal) A|^2 \da + \tfrac{1}{2} \left(\tfrac{1}{2} - \tfrac{1}{n-1} \right) \tfrac{n-2}{n-1} \int_\Sigma H \tr h^\intercal (L_\Sigma \tr h^\intercal )\da
\end{align*}
where in the last line we use $\int_\Sigma \zeta \tr h^\intercal\da=0$ again and \eqref{eq:hiddenbd2}. Equating \eqref{eq:bulk-sign} with the previous integral yields
\begin{align} \label{eq:bulk-sign}
\begin{split}
&\tfrac{1}{2} (n-1)(n-2) n \Lambda b^2 \int_\Sigma H \dvol\\
&= \int_\Sigma H  |A'(h) - \tfrac{1}{n-1} (\tr h^\intercal) A|^2 \da\\
&\quad +\tfrac{1}{2} \left(\tfrac{1}{2} - \tfrac{1}{n-1} \right) \tfrac{n-2}{n-1} \int_\Sigma H \tr h^\intercal (L_\Sigma \tr h^\intercal )\da.
\end{split}
\end{align}
The $H$-non-degenerate boundary condition implies that the last integral is strictly greater than zero, unless $\tr h^\intercal=0$. On the other hand, because of $\Lambda <0$ the left hand side is strictly less than zero, unless $b=0$.  Therefore,  we conclude $h^\intercal=0$, $A'(h)=0$ and $b=0$, which completes the proof.
 \end{proof}

We prove the following more detailed version of Theorem~\ref{th:general}.

\begin{manualtheorem}{\ref{th:general}$^\prime$}\label{th:general'}
Let $n\ge 4$ and $\Omega$ be $n$-dimensional with smooth boundary $\Sigma$ satisfying $\pi_1(\Omega, \Sigma)=0$. Then the boundary map $\Pi :\widehat {\mathcal M}_{-}^{k,\alpha}(\Omega) /\mathscr D^{k+1,\alpha}(\Omega)$ is regular on an open dense subset $\mathcal W$. 

More specifically, for each $\llbracket\bar g\rrbracket\in \mathcal W$, there exists $0<\epsilon \ll 1$, a neighborhood $\mathcal U\subset \widehat{ \mathcal M}_{-}^{k,\alpha}(\Omega) /\mathscr D^{k+1,\alpha}(\Omega)$ of $\llbracket \bar g\rrbracket $, and a neighborhood $\mathcal V\subset \mathcal S_1^{k,\alpha}(\Sigma)\times \C^{k-1,\alpha}(\Sigma)$ of $([\bar g^\intercal], H_{\bar g})$ such that for each $(\gamma, \phi)\in \mathcal V$,  $\Pi^{-1} (\gamma, \phi)$ is a smooth curve $g_s$ in $\mathcal U$ for  $ |s|<\epsilon $ satisfying
\[
	\Ric_{g_s} = (1+s)\Lambda \left( \frac{V_{\bar g}}{V_{g_s}}\right)^{\frac{2}{n}} g_s,
\]
and  $g_s$ is not  isometric to $g_{s'}$ for all $s\neq s'$, where $\Lambda$ is the constant such that $\Ric_{\bar g} = (n-1)\Lambda \bar g$.
\end{manualtheorem}

\begin{proof}
By Theorem~\ref{th:trivial}, $\overline T^\Ga$ is a local diffeomorphism at $(g, 1)$ for $g$ in an open dense subset  $\mathcal W_0 \subset \widehat {\mathcal M}^{k,\alpha}_-(\Omega)$. We let $\mathcal W =\mathcal W_0 /\mathscr D^{k+1,\alpha}(\Omega)$. Let $\bar g\in \mathcal W_0$. We show that the linearized boundary map $D\Pi$ at $\llbracket \bar g\rrbracket$ is surjective and its kernel is one-dimensional, and thus $\Pi$ is regular at $\llbracket \bar g\rrbracket$. To simplify notation in this proof, we denote  $\mathscr M:=\widehat{ \mathcal M}_{-}^{k,\alpha}(\Omega) /\mathscr D^{k+1,\alpha}(\Omega)$. By local representation of $\mathscr M$ in the harmonic gauge with respect to $\bar g$ (see Remark~\ref{re:harmonic}), its tangent space at $\bar g$ is given by 
\[
	T_{\bar g} \mathscr M = \Big\{ h\in \C^{k,\alpha}(\Omega): \Ric'(h) - (n-1) h - (n-1) \Lambda b \bar g = 0  \mbox{ and } \beta h = 0  \mbox{ in } \Omega\Big \}
\]	  
where the constant $b$ depends on $h$, explicitly $b = \frac{1}{n(n-1)\Lambda} R'(h)$. Since $\overline L^\Ga$ is an isomorphism and in particular surjective onto the Anderson boundary data, it implies that $D\Pi$  is surjective. The kernel of $D\Pi$ is given by $T_{\bar g} \mathscr M  \cap \big\{ h: h^\intercal = \frac{1}{n-1} (\tr h^\intercal) \bar g^\intercal \mbox{ and } H'(h)=0  \mbox{ on } \Sigma\big\}$ and is equal to the space $\{ (h, b):  \overline L^\Ga(h, b) = (0, 0, 0, 0, 0, a) \mbox{ for some } a\in \mathbb R\}$, which is one-dimensional  because $\overline L^\Ga$ is an isomorphism. This completes the proof of the first part.

We discuss the second part of the theorem. Because $\overline T^\Ga$ is a local diffeomorphism,  there exists a neighborhood $\overline {\mathcal U}\subset \mathcal M^{k,\alpha}(\Omega)\times \mathbb R$ of $(\bar g, 1)$ so that $\overline T^\Ga$ is diffeomorphic onto its image, say  $\overline {\mathcal V}$, a neighborhood of $T^\Ga(\bar g, 1) = (0, 0, [g^\intercal], H_{\bar g}, 2V_{\bar g})$. For each $(0, 0, \gamma, \phi, 2s^{\frac{n}{2}}V_{\bar g})\in \overline {\mathcal V}$ where $|s-1|$ small, there is a unique solution $(g_s, \eta_s)\in \overline{\mathcal U}$ such that 
\[
	T^\Ga (g_s, \eta_s) = (0, 0, \gamma, \phi, 2s^{n/2}V_{\bar g}).	
\]
By shrinking $\overline{\mathcal U}$ smaller if necessary, we have $\beta g_s =0$ and thus 
\begin{align*}
	\Ric_{g_s} &= (n-1) \Lambda \eta_s g_s\quad  \mbox{ in } \Omega\\
	2\eta_s^{\frac{n}{2}} V_{g_s}&=2s^{n/2}V_{\bar g},
\end{align*}
From there, we can solve $\eta_s$ and find that 
\[
	\Ric_{g_s}=(n-1) \Lambda s \left(\frac{V_{\bar g}}{V_{g_s}}\right)^{\frac{2}{n}} g_s. 
\]
We note that $g_s$ is not isometric to $g_{s'}$ for any $s\neq s'$ by uniqueness of $(g_s, \eta_s)$ in $\overline{\mathcal U}$. It is straightforward to verify that $\overline{\mathcal U}$ and $\overline{\mathcal V}$ give rise to the desired neighborhoods in the theorem.

\end{proof}

\section{Examples with non-degenerate boundary}\label{se:example}

We give examples of $(\Omega, \bar g)$ that have non-degenerate boundary. Since those examples have constant mean curvature boundary, they also satisfy the $H$-non-degenerate boundary condition.

For each $n\ge 4$,  $\Lambda \in \mathbb R$, and $m>0$, we define the function $f(r)=1-2mr^{3-n}-\Lambda r^2$. 
\begin{itemize}
\item If $\Lambda=0$,  $f(r)$ has a unique positive root $r_0 = (2m)^{\frac{1}{n-3}}$;
\item If $\Lambda< 0$, $f(r)$ has a unique positive root $r_0$; 
\item If $\Lambda>0$, $f(r)$ attains a unique maximum at  $r^*=\left(\frac{\Lambda}{m(n-3)}\right)^{1/(1-n)}$, and we require $f(r^*)>0$; equivalently, 
 \begin{align}\label{eq:Lambda}
 m^2\Lambda^{n-3}<\frac{(n-3)^{n-3}}{(n-1)^{n-1}}.
 \end{align} 
 Under this assumption,  $f(r)$ has two positive roots $r_0, r_1$ satisfying $r_0<r^*<r_1$. For the rest of this section, for the case $\Lambda>0$ we always assume \eqref{eq:Lambda}.
 \end{itemize}
 
Using $f(r)$, we define a Lorentzian metric as follows:
\[
\mathbf g= -f(r)dt^2+\frac{1}{f(r)}dr^2+r^2 g_{S^{n-2}},
\]
where $t\in \mathbb R$ and  $r\in (r_0, r_1)$ (allowing $r_1 = \infty$ for $\Lambda \le 0$). Here and below, $g_{S^j}$ denotes the round metric of the  $j$-dimensional unit sphere. When $\Lambda = 0$, $\Lambda < 0$, and $\Lambda > 0$,  this metric represents respectively a Schwarzschild, anti de-Sitter Schwarzschild, and de-Sitter Schwarzschild spacetime metric that satisfies the (spacetime) Einstein equation $\Ric_{\mathbf g} = (n-1) \Lambda \mathbf g$. 

By changing the negative signature of $\mathbf g$ to a positive signature, we obtain the corresponding \emph{Euclidean} Schwarzschild,  anti-de-Sitter Schwarzschild metric, or de-Sitter Schwarzschild:
\begin{align*}
g= f(r)dt^2+\frac{1}{f(r)}dr^2+r^2 g_{S^{n-2}}.
\end{align*}
The metric appears to be singular at $r_0$, but one can ``desingularize'' the coordinate singularity at $r_0$ by quotienting $t\in \mathbb{R}$ to a circle of appropriate radius $2\ell$, denoted by $S^1_{2\ell}$, where
\[
\ell =\frac{1}{(n-3)r_0^{-1}-(n-1)\Lambda r_0}.
\]
(Note that $\ell = \ell_{\Lambda, n, m}$  depends on $\Lambda, n, m$. For example $\ell =\frac{1}{n-3} (2m)^{\frac{1}{n-3}}$ for $\Lambda =0$.) Then we can extend $g$ to be defined on a manifold that is topologically $S^1 \times  [r_0, r_1)\times S^{n-1}$. We refer the details to Appendix~\ref{se:remove}. In summary,
\begin{itemize}
\item When $\Lambda \le 0$, we identify $S^1 \times  [r_0, \infty)$ with $\mathbb R^2$ and let $M=\mathbb R^2\times S^{n-2}$. Then  $(M, g)$ is  Einstein with $\Ric_g = (n-1)\Lambda g$ and is complete without boundary.
\item When $\Lambda > 0$, we identify $S^1 \times  [r_0, r_1)$ with an open disk $B^2$ and let $M=B^2\times S^{n-2}$. Then  $(M, g)$ is  Einstein satisfying $\Ric_g = (n-1)\Lambda g$ and is complete, except  a conical singularity at $\partial B^2\times S^{n-2}$.
\end{itemize}

Consider the compact subset $\Omega_{s} \subset M$ as the sublevel set of $r$:
\begin{align*}
\Omega_{s}=\left\{r_0\leq r\leq s \right\},
\end{align*}
where $\Omega_{s}$ is topologically $B^2\times S^{n-2}$ and the  boundary $\Si_{s}$ is topologically $S^1\times S^{n-2}$. In particular, the topological condition $\pi_1(\Omega_s, \Sigma_s)=0$ holds.

\begin{proposition}
For each $n\ge 4$,  $\Lambda \in \mathbb R$ (satisfying \eqref{eq:Lambda} if $\Lambda>0$), and $m>0$,  the compact subset $(\Omega_s, g)$ has a non-degenerate boundary if and only if $s\ge r_0$ satisfies
\begin{align}\label{eq:radius}
	s^2> 4 (n-3)\ell^2 f(s). 
\end{align}
\end{proposition}
\begin{proof}

The metric induced on $\Sigma_s$ is given by $g^\intercal =4\ell^2 f(s)g_{ S^{1}}+s^2 g_{ S^{n-2}}$. 
By direct computation, the induced scalar curvature $R_{\Sigma_s}$ and the mean curvature $H_{\Sigma_s}$ are both constant, and  
\[	
	R_{\Sigma_s} = \frac{(n-3)(n-2)}{s^2}.
\]
Consider the operator $L_{\Sigma_s} v: = -\Delta_{\Sigma_s} v  - \tfrac{1}{n-2} R_{\Sigma_s} v$. It is obvious that the first eigenfunction of $L_{\Sigma_s}$ is constant with the first eigenvalue $\lambda_1 = -\frac{1}{n-2} R_{\Sigma_s} = -\frac{n-3}{s^2}  < 0$. 

Because $H_{\Sigma_s}$ is constant, the non-degenerate boundary condition is equivalent to showing that the second eigenvalue of $L_{\Sigma_s}$ is strictly positive. Equivalently, we show that  any eigenfunction of $L_{\Sigma_s}$ with a non-positive eigenvalue must be constant.

Let $v$ solve $L_{\Sigma_s} v = \lambda v$ and $\lambda \le 0$. Denote the coordinates of $S^{n-2}$ by $\theta = (\theta_1, \dots, \theta_{n-2})$. For $k = 0, 1, 2, \dots$ and integers $j$ with $|j|\le k$,  let $\{ Y_j^k(\theta) \}$ be a basis of spherical harmonics on $S^{n-2}$ of degree $k$ such that $\Delta_{S^{n-2}} Y_j^k = -k (k + n -3 ) Y^k_j$. Then we can express $v$ as $v (t, \theta) = \sum_{k, j} a_{jk}(t) Y^k_j(\theta)$. Note that 
\[
 \Delta_{\Sigma_s} v =\frac{1}{4\ell^2 f(s)}\D_{S^1}v+\frac{1}{s^2}\D_{S^{n-2} }v. 
 \]
 We calculate 
 \begin{align*}
 	0 &= \Delta_{\Sigma_s} v + \frac{1}{n-2} R_{\Sigma_s} v + \lambda v\\
	& =\sum_{k, j} \left (\frac{1}{4\ell^2f(s)} \Delta_{S^1} a_{jk}  + \left( \frac{- k(k+n-3) + (n-3) }{s^2} + \lambda \right) a_{jk} \right) Y^k_j (\theta).
 \end{align*}	
 It implies, for each $j, k$, 
\begin{align} \label{eq:circle}
	\Delta_{S^1} a_{jk}  = 4\ell^2 f(s)\left( \frac{ k(k+n-3) - (n-3)}{s^2} - \lambda \right) a_{jk}.
\end{align}
Note that the eigenvalues of $\Delta_{S^1}$ are $-i^2$ for $i=0, 1, 2, \dots$. So  in particular
\[
\frac{ k(k+n-3) - (n-3)}{s^2} - \lambda\le 0.
\]
Since the left hand side is positive for $k>0$, we must have  $a_{jk}\equiv 0$ for all $k >0$. When $k=0$, the condition  \eqref{eq:radius} is chosen to ensure that the eigenvalue in \eqref{eq:circle} satisfies
 \[
 	 4\ell^2 f(s) \left( \frac{ - (n-3)}{s^2} - \lambda \right) \ge  -4\ell^2 f(s)  \frac{ (n-3)}{s^2}> -1.
 \]
Thus, the eigenvalue in \eqref{eq:circle} must be zero, and hence $\lambda = -\frac{n-3}{s^2}$ and $v$ is constant. 

It is also direct to see that if $s$ doesn't satisfy \eqref{eq:radius}, then there exists $\lambda\le 0$ such that  $4\ell^2 f(s) \left( \frac{ - (n-3)}{s^2} - \lambda \right)=-1$, and thus the second eigenvalue of  $L_{\Sigma_s}$ is equal to or less than zero. It implies that $\Omega_s$ doesn't satisfy the non-degenerate boundary condition.

\end{proof}

We discuss some examples of the range of $s$ for which \eqref{eq:radius} holds in the following lemma. It turns out to have a complicated dependence on  $\Lambda$,  $m$, and  $n$. 
\begin{lemma}
\begin{enumerate}
\item The case $\Lambda=0$:  \eqref{eq:radius} holds for all $s$ close to $r_0$ and for all $s$ sufficiently large. 
\item The case $\Lambda<0$:  \eqref{eq:radius} holds for all $s$ close to $r_0$. In this case \eqref{eq:radius}  holds for all $s$ sufficiently large if and only if  $\Lambda> -\frac{1}{4(n-3)\ell^2}$.
\item  The case $\Lambda>0$ satisfying \eqref{eq:Lambda}: \eqref{eq:radius} holds for all $s$ close to $r_0$ and to $r_1$. 
\item When $n=4$ and $\Lambda$ is sufficiently close to zero,  \eqref{eq:radius} holds for all $s\ge r_0$.
\end{enumerate}
\end{lemma}
\begin{proof}
 Define 
\[
 p(s) : =s^2- 4 (n-3)\ell^2 f(s) =  (1+ 4(n-3) \ell^2 \Lambda) s^2 + 4(n-3) \ell^2 (2m s^{3-n} -1). 
\] 
Then the condition  \eqref{eq:radius} becomes $p(s)>0$. Because $r_0$ is a root of $f(s)$, it is obvious that $p(s) >0$ for $s$ near $r_0$ for any case of $\Lambda$. For the same reason,  $p(s)>0$ for $s$ near $r_1$ when $\Lambda>0$.  In the case $\Lambda\le 0$, then $p(s) >0$ for $s$ sufficiently large is equivalent to that the coefficient of $s^2$ is positive, that is, 
\[
1+ 4(n-3) \ell^2 \Lambda>0 \quad \mbox{ equivalently} \quad  \Lambda> -\frac{1}{4(n-3)\ell^2}.
\]

One can compute that  
\[
	p'(s) = 2(1+4(n-3) \ell^2 \Lambda) s - 8 (n-3)^2 \ell^2 m s^{2-n}.
\]
There is a unique positive critical point 
\[
	s_c = \left(\frac{4(n-3) \ell^2 m}{1+ 4(n-3) \ell^2 \Lambda}\right)^{\frac{1}{n-1}}.
\]
For either case that $s_c\le r_0$ or $p(s_c)>0$, we can conclude that $p(s)>0$ for all $s\ge r_0$.  We only check the later case for $\Lambda =0$ and $n=4$. In this case, $\ell = 2m, s_c = 16^{\frac{1}{3}} m$, and 
 \[
 	p(s) = s^2 + 4\ell^2 (2m s^{-1} - 1) = 3\cdot 16^{\frac{2}{3}} m^2 - 16 m^2 > 0.
 \]
\end{proof}

\appendix

\section{Linearized formulas}\label{se:va}     

Let $(\Omega, g)$ be a compact Riemannian manifold with smooth boundary $\Sigma$. Let $g(s)$ be a smooth family of Riemannian metrics with $g(0)= g$ and $g'(0)=h$. In what follows, we  use the prime notation $'$ to denote the linearized operator. For example,  we denote $A'|_g(h)  =\left. \ds \right|_{s=0} A_{g(s)}$ for the linearized second fundamental form at $g^\intercal$ on $\Sigma$ and $\Ric'|_g(h) =\left. \ds \right|_{s=0} \Ric_{g(s)}$ for the linearized Ricci curvature at $g$ on $\Omega$. 

Consider a local frame $\{ e_0, e_1, \dots, e_{n-1}\}$ in a collar neighborhood of $\Sigma\subset \Omega$ such that $e_0$ is the parallel extension of $\nu_g$ along itself and $e_1, \dots, e_{n-1}$ are tangent to $\Sigma$. 
 The linearized unit normal, second fundamental form, and mean curvature of the boundary $\Sigma$ are given by the following formulas:   for tangential direction $a, b, c \in \{ 1, \dots, n-1\}$ on $\Sigma$,         
\begin{align}
	\nu '|_g(h)&=-\tfrac{1}{2} h (\nu_g, \nu_g) \nu_g - g^{ab} \omega(e_a) e_b \notag\\
	A'|_g(h) &= \tfrac{1}{2} (\mathscr L_{\nu_g} h)^\intercal - \tfrac{1}{2} \mathscr L_{\omega} (g^\intercal) - \tfrac{1}{2} h(\nu_g, \nu_g) A_g \label{eq:A}\\
	H'|_g(h)&= \tfrac{1}{2} \nu_g(\tr_g h^\intercal) - \Div_\Sigma \omega - \tfrac{1}{2} h(\nu_g, \nu_g)H_g \label{eq:H}
 \end{align}         
 where $\omega(e_a) = h(\nu_g, e_a)$.
 See the derivation at, e.g. \cite[Lemma 2.1]{An-Huang:2022}.

To express $\Ric|_g'(h)$, we define the Killing operator $\mathcal D_g$ and the Bianchi operator $\beta_{g} $: for any vector field $X$ and symmetric $2$-tensor $h$, 
\begin{align} \label{eq:KB}
	\mathcal D_g X = \tfrac{1}{2} \mathscr L_X g \quad\mbox{ and } \quad \beta_{g} h = - \Div_{g} h + \tfrac{1}{2} d \tr_{g} h.
\end{align}
By direct computation, 
\begin{align} 
{ \beta_{g}  \mathcal D_g X} &{= -\tfrac{1}{2} (\Delta_g X +  \Ric_{g} (X, \cdot) )}\label{eq:Laplace0}
\end{align}
where $\Delta_g$ denotes the trace of the Hessian operator of $g$. The formal $\mathcal L^2$ adjoint operators become
\begin{align}
\label{eq:KB-adjoint}
\begin{split}
	\mathcal D_g^* h &=-\Div_g h \quad \mbox{and} \quad \beta^*_g X = \tfrac{1}{2} \left(\mathscr L_X g - (\Div_g X )g\right)\\
\mathcal D_g^* \beta_{g}^* X  &=   {-\tfrac{1}{2} (\Delta_g X +  \Ric_{g} (X, \cdot) )}.
\end{split}
\end{align}
When $g = \bar g$ is Einstein such that $\Ric_{\bar g} = (n-1)\Lambda \bar g$, we have
{\begin{align*}
\begin{split}
	\beta  \mathcal D X &= -\tfrac{1}{2} (\Delta X + (n-1)\Lambda X)\\
	 \mathcal D^* \beta^* X&= -\tfrac{1}{2} (\Delta  X + (n-1)\Lambda X). 
\end{split}
\end{align*}
Above and in the following, we will omit the subscript $\bar g$ when evaluating at the background Einstein metric $\bar g$. 
}

By direct computations, 
\begin{align}
\begin{split}\label{equation:Ricci}
	(\Ric'|_g(h))_{ij} &=-\tfrac{1}{2} g^{k\ell} h_{ij;k\ell} + \tfrac{1}{2} g^{k\ell} (h_{i k; \ell j} + h_{jk; \ell i} ) \\
	&\quad - \tfrac{1}{2} (\tr_g h)_{;ij}  + \tfrac{1}{2} (R_{i\ell} h^\ell_j + R_{j\ell} h^\ell_i )\\
	&\quad - R_{ik\ell j} h^{k\ell} \qquad \qquad  \mbox{ for all } i, j =0,1,\dots, n-1.
\end{split} 
\end{align}
By letting $g = \bar g$ be Einstein such that $\Ric_{\bar g} = (n-1)\Lambda \bar g$, we can simplify the above expression  
\begin{align*}
	(\Ric'(h))_{ij} 
	&= -\tfrac{1}{2} (\Delta h)_{ij}  -(\mathcal D \beta  h )_{ij}+ (n-1)\Lambda h_{ij}-R_{ik\ell j} h^{k\ell}.
\end{align*}
Therefore, the linearized Ricci and its formal $\mathcal L^2$ adjoint operator are given by
\begin{align}
	(\Ric'(h))_{ij}  -  (n-1) \Lambda h_{ij}&=-\tfrac{1}{2} (\Delta h)_{ij}  -(\mathcal D \beta h)_{ij} - R_{ik\ell j} h^{k\ell} \notag \\
	((\Ric')^*(\gamma))_{ij}  - (n-1) \Lambda \gamma_{ij} &= -\tfrac{1}{2} (\Delta \gamma)_{ij}  + (\beta^* \Div \gamma)_{ij} - R_{ik\ell j} h^{k\ell}. \label{eq:liEin}
\end{align} 
Note that 
\begin{align}\label{eq:divergence}
	\Div\big( (\Ric')^*(\gamma)  - (n-1) \Lambda \gamma \big)=0 \quad \mbox{ in } \Omega
\end{align}
because for any $X\in \C^\infty_c(\Omega)$
\begin{align*}
	&\int_\Omega X \cdot \Div\big(  (\Ric')^*(\gamma)  - (n-1) \Lambda \gamma \big)  \dvol\\
	 &=-\tfrac{1}{2} \int_\Omega  \mathscr L_X g\cdot \big(  (\Ric')^*(\gamma)  - (n-1) \Lambda \gamma \big)  \dvol\\
	&= -\tfrac{1}{2} \int_\Omega \big( \Ric'(\mathscr L_X g) - (n-1) \Lambda\mathscr L_X g \big)  \cdot\gamma \dvol = 0.
\end{align*}

\section{Banach manifolds}
We verify the Banach manifold structure of the function spaces used in this paper.   Let $\Sigma$ be an $(n-1)$-dimensional closed manifold.               
\begin{lemma}\label{le:Banach}
Let $\mathcal S_1^{k,\alpha}(\Sigma)$ denote the space of $\C^{k,\alpha}$ Riemannian metrics on $\Sigma$ whose determinant is $1$ (with respect to a fixed background metric $\bar \gamma$). Then $\mathcal S_1^{k,\alpha}(\Sigma)$ is a Banach manifold, and its tangent space at an arbitrary $\gamma$ is $T_\gamma \mathcal S_1^{k,\alpha} = \{ \tau : \tr_{\gamma} \tau = 0\}$. 
\end{lemma}
\begin{proof}
Let $f: \mathcal M^{k,\alpha}(\Sigma)\to \C^{k,\alpha}(\Sigma)$ be $f(\gamma) = |\gamma|^\frac{1}{2} = \frac{\da_{\gamma}}{\da_{\bar \gamma}}$. The linearization at $\gamma$ is $f'(h) = \tfrac{1}{2} (\tr_\gamma h) \frac{\da_{\gamma}}{\da_{\bar \gamma}}$. It is clear that $f'(h)$ is surjective and  the kernel splits (as  $h$ can be decomposed into the traceless part and the conformal part). By the Submersion Theorem, the level set $f^{-1}(1)$ is a submanifold, and the tangent space is as described.

\end{proof}

     \section{Local Riemannian isometries}

We let $n\ge 2$ (only in this section) and let  $\Omega$ be an $n$-dimensional, compact, simply-connected manifold possibly with connected boundary and $(S_\Lambda, g_\Lambda)$ be the $n$-dimensional spaceform of sectional curvature $\Lambda$ and $g_\Lambda$ is smooth. We discuss how isometric immersions of constant sectional spaces $(\Omega, g)$ into spaceforms depend on~$g$. 

\begin{proposition}\label{pr:immersion}
Suppose $F_1: (\Omega, g_1)\to (S_\Lambda, g_{\Lambda})$ is $\C^{k+1,\alpha}$ local Riemannian isometry. Fix $p\in \Omega$. Then there exists $0<\epsilon_0\ll 1$ and $C>0$ such that if $\epsilon<\epsilon_0$ and $g_2\in \C^{k,\alpha}(\Omega)$ has constant sectional curvature $\Lambda$,  $\| g_1 - g_2\|_{\C^{k, \alpha}(\Omega)}<\epsilon$ and $dF_1|_p:(T_p\Omega, g_2)\to (T_{F_1(p)}S_\Lambda, g_{\Lambda})$ is a linear isometry, then there exists a unique $\C^{k+1,\alpha}$ local Riemannian isometry $F_2: (\Omega, g_2)\to (S_\Lambda, g_\Lambda)$ whose $1$-jet coincides with that of $F_1$, i.e. $F_1(p)= F_2(p)$, $dF_1|_p = dF_2|_p$, such that
\begin{align}\label{eq:estimate}
	\| F_1 - F_2 \|_{\C^{k+1, \alpha}(\Omega)}<C\epsilon.
\end{align}
\begin{remark}

\end{remark}
\margin{remark?}
\end{proposition}
\begin{proof}
The existence  and uniqueness of the isometric  immersion $F_2$ is standard, see, e.g. \cite[Theorem 5.6.7]{Petersen:2016}. We just need to show \eqref{eq:estimate}. By compactness of $\Omega$, there exist finitely many points $p_i \in \Omega$ for $i=1,\dots N$ where $p_1 = p$ and normal neighborhoods $U_i$ of $p_i$ with respect to both $g_1, g_2$, provided $\epsilon_0\ll 1$. By relabelling if necessary, we can assume $U_i\cap U_j\neq \emptyset$ for some $1\le j< i$. Denote the exponential maps at $p$ with respect to $g_1, g_2$ at $p$  by $\exp^1_p$ and $\exp^2_p$ respectively.  We will construct $F_2$ starting on $U_1$. Note that $F_1$ on $U_1$ is uniquely determined by, for $q\in U_1$, 
\[
	F_1(q) = \exp_{F_1(p_1)}\circ (dF_1)|_{p_1} \circ (\exp_{p_1}^1)^{-1}(q). 
\]
Since $g_2$ has constant sectional curvature, there is a unique local Riemannian isometry $F_2: (\Omega,\gamma) \to S_{\Lambda}$ that has the same 1-jet at $p_1$ as $F_1$. Note that $F_2$ is at least $\C^{k-1}$ because $F_2$ is uniquely determined by, for $q\in U_1$, 
\[
	F_2(q) = \exp_{F_2(p_1)}\circ (dF_2)|_{p_1} \circ (\exp_{p_1}^2)^{-1}(q)
\]
and $\exp_{p_1}^2\in \C^{k-1}$. 
Since the 1-jets of $F_1$ and $F_2$ coincide at $p_1$ and  $\|(\exp_{p_1}^1)^{-1}-(\exp_{p_1}^2)^{-1} \|_{\C^{k-1}(U_1)}<C\epsilon$ for some $C$ depending only on $g_1$ by smooth dependence of geodesics, we obtain
\[
	\| F_1 - F_2 \|_{\C^{k-1}(U_1)} < C\epsilon
\]
where $C$ depends on $g_1, F_1$ and $U_1$.

We also note that in the above argument, even if the 1-jets of $F_1$ and $F_2$ do not coincide, we can still obtain the same estimate, as long as $|F_1(p_1)-F_2(p_1)| + |dF_1|{p_1} - dF_2|{p_1}|< C\epsilon$ for some $C$ depending only on $g_1$ and $F_1$. Then we repeat the argument for each successive $U_i$ and obtain the estimate on the entire $\Omega$ (we can enlarge $C$ if necessary):
\begin{align}\label{eq:estimate2}
\| F_1 - F_2 \|_{\mathcal{C}^{k-1}(\Omega)} < C\epsilon
\end{align}
where $C$ depends on $g_1, F_1$ and $U_1$. 

While $F_2$ is at least $\C^{k-1}$ from the above argument,  in fact  $F_2\in \mathcal{C}^{k+1,\alpha}(\Omega)$ by \cite[Theorem 2.1]{Taylor:2006}, and we will adapt the argument there to upgrade the estimate \eqref{eq:estimate2} to \eqref{eq:estimate}. Given $x\in S_\Lambda$, pick local harmonic coordinates $(u_1, \dots, u_n)$ on a neighborhood $V$ of $x$. Since $g_\Lambda\in \C^\infty$, each $u_i\in \mathcal{C}^\infty$. We assume $U\subset \Omega$ is sufficiently small so that both  $F_1(U)$ and $F_2(U)$ are contained in $V$.  Since $F_1$ and $F_2$ are local isometries, $v_1:=u\circ F_1$ and $v_2:=u\circ F_2$ are also harmonic on $U$ with respect to $g_1$ and $g_2$ respectively. Therefore, estimating $F_1-F_2$ is equivalent to estimating $v_1 - v_2$. From \eqref{eq:estimate2}, we have $\| v_1 - v_2 \|_{\mathcal{C}^{k-1}(U)} < C\epsilon$, where $C$ depends on $g_1, F_1, U_1, u$.  

 Subtracting the harmonic equations of $v_1$ and $v_2$ gives an elliptic equation for $v_1 - v_2$ with the inhomogeneous terms of the form $|g_1 - g_2||\partial^2 v_1|+ |g_1 \partial g_1 - g_2 \partial g_2| |\partial v_1|$. Let $U_0 \subset U$ be a proper open subset.  The interior Schauder estimate says
\[
	\| v_1 - v_2 \|_{\C^{k+1,\alpha}(U_0)}  \le C\left( \| \mbox{inhomogeneous terms} \|_{\C^{k-1,\alpha}(U)} +\| v_1 - v_2 \|_{\C^0(U_0)} \right).
\]
The right hand side can be bounded by $\| g_1 - g_2 \|_{\C^{k,\alpha}}$, $\|u\circ  F_1\|_{\C^{k+1,\alpha}}$, and $\| v_1 - v_2 \|_{\C^0(U_0)}$, and thus we  obtain $\| v_1 - v_2 \|_{\C^{k+1,\alpha}(U_0)} \le C\epsilon$. It is then direct to see that the local estimate can be patched up to obtain the global estimate, which implies \eqref{eq:estimate}.

\end{proof}

Let $\mathcal N_\Lambda^{k,\alpha} (\Omega)$ denote the space of $\C^{k,\alpha}$ metrics on $\Omega$ with constant sectional curvature $\Lambda$. Note that when $n=3$, $\mathcal N_\Lambda^{k,\alpha} (\Omega) = \mathcal M_\Lambda^{k,\alpha} (\Omega) $, the space of Einstein metrics. Let $F_0 : \Omega \to S_\Lambda$ be a smooth immersion. Recall $\mathrm{Imm}^{k+1,\alpha}_p(\Omega, S_\Lambda)  =\big \{ F:\Omega \to S_\Lambda :  \mbox{$F$ is a $\C^{k+1,\alpha}$  immersion with } F(p) = F_0(p), dF|_p = d F_0|_p \big\}$ and $\mathscr D_p^{k+1,\alpha} (\Omega) = \left\{ \psi \in \mathscr D^{k+1,\alpha}(\Omega): \psi(p)=p, \,  d\psi|_p=\mathrm{Id} \right\}$. 

\begin{lemma}\label{le:embedding}
There is a local diffeomorphism
\[
\Psi: \mathrm{Imm}_p^{k+1,\alpha} (\Omega, S_\Lambda)/\mathscr D_p^{k+1,\alpha}(\Omega) \to \mathcal N_\Lambda^{k,\alpha} (\Omega)/\mathscr D^{k+1,\alpha}(\Omega).
\] 
\end{lemma}
\begin{proof}
Define the map 
\[
\Psi :\mathrm{Imm}_p^{k+1,\alpha} (\Omega, S_\Lambda)/\mathscr D_p^{k+1,\alpha}(\Omega) \to \mathcal N_\Lambda^{k,\alpha} (\Omega)/\mathscr D^{k+1,\alpha}(\Omega)
\]
 by sending the equivalent class $\llbracket F \rrbracket$ to $\llbracket F^*(g_\Lambda|_{F(\Omega)})\rrbracket$. It is obvious that $\Psi$ is well-defined and smooth, and $\Psi(\llbracket F \rrbracket)$ is independent of the choice of representative~$F$. 

We construct a local inverse for $\Psi$ as follows. Fix $F\in \mathrm{Imm}_p^{k+1,\alpha}(\Omega, S_\Lambda)$ and let $g = F^* (g_\Lambda|_{F(\Omega)})$. As in Lemma~\ref{le:gauge} and Remark~\ref{re:harmonic}, let $\mathcal U_\Lambda$ be a neighborhood of $g$ in $\mathcal{N}_\Lambda^{k,\alpha} (\Omega)$. Let  $\gamma$ be the unique harmonic representative in  $\llbracket \gamma \rrbracket\in \mathcal U_\Lambda/\mathscr D_p^{k+1,\alpha}(\Omega) $. Our goal is to identify $\gamma$ with some $G\in  \mathrm{Imm}_p^{k+1,\alpha} (\Omega, S_\Lambda)$. Note that we cannot simply take $G$ to be a local Riemannian isometry of $(\Omega, \gamma)$ into the spaceform because $dF_0$ need not be a linear isometry between $(T_p\Omega, \gamma)$ and $(T_{F_0(p)} S_\Lambda, g_\Lambda)$ and thus such $G$ does not have the same $1$-jet as $F_0$.  Below we describe the procedure to find another representative in $\llbracket \gamma\rrbracket $ that has such linear isometry and show that the procedure depends smoothly on $\gamma$ and is well-defined.

Let  $S_\gamma:   (T_p \Omega, g) \to (T_p\Omega, \gamma) $ be a linear isometry that depends smoothly on $\gamma$.  One can find $\psi_1$ sufficiently close to $\mathrm{Id}_\Omega$ such that $\psi_1(p)=p$ and $d\psi_1|_p = S_\gamma$.  (For instance, construct $\psi_1$ locally around a coordinate chart of $p$ to be $S_\gamma$ at $p$, and then transition $\psi_1$ to be the identity map near the boundary.) Let $\gamma_1 = \psi_1^* \gamma$, and thus $(T_p\Omega, \gamma_1)$ and $(T_p \Omega, g)$ are linearly isometric, and hence $dF_0: (T_p\Omega, \gamma_1)\to (T_{F_0(p)}S_\Lambda, g_{\Lambda})$ is a linear isometry.  By Proposition~\ref{pr:immersion}, there is a unique isometric immersion $ G_1: (\Omega, \gamma_1)\to (S_\Lambda, g_\Lambda)$ having the same 1-jet as $F_0$ at $p$. By construction, the map from $\gamma$ to $G_1$ is smooth. We define $\Psi^{-1}(\llbracket \gamma \rrbracket)=\llbracket G_1 \rrbracket$.

We verify that the inverse map is independent of the choice of the diffeomorphism $\psi_1$. Let another $\psi_2$ satisfy $\psi_2(p)=p$, $d\psi_2|_p = S_\gamma$, and $\gamma_2 := \psi_2^* \gamma $. Denote the corresponding immersion by $G_2: (\Omega, \gamma_2)\to (S_\Lambda, g_\Lambda)$. It is direct to verify that $G_1\circ \psi_1^{-1}\circ \psi_2$ is also an immersion of $(\Omega, \gamma_2)$ that has the same 1-jet as $G_2$ at $p$. By the uniqueness of such immersion, we must have $G_1\circ \psi_1^{-1}\circ \psi_2 = G_2$, which implies $G_1$ and $G_2$ are in the same equivalence class.

\end{proof}

\section{Global rigidity for conformal Cauchy boundary}\label{se:global}
 We give the  global rigidity for the (nonlinear) conformal Cauchy boundary data to motivate the linear version Theorem~\ref{th:infinitesimal}. It is of independent interest and is not used elsewhere in the paper.

The Fundamental Theorem of Surfaces asserts that if $\hat F, F: \Sigma \to \mathbb R^3$ are two immersions with the same Cauchy boundary condition, i.e. the same induced metric and second fundamental form, then $\hat F$ and $F$ are congruent. We extend this fundamental theorem to accommodate the conformal Cauchy boundary condition.

\begin{proposition}\label{pr:global}
Let $\Sigma$ be a closed surface and $\hat F, F: \Sigma\to \mathbb R^3$ be two immersions. Suppose there is a $\C^2$ function $u>0$ on $\Sigma$ such that 
\begin{align*}
\hat \gamma = u \gamma\quad \mbox{ and } \quad \hat A = u A
\end{align*}
where $\hat \gamma, \gamma$ are the induced metrics and $\hat A, A$ are the induced second fundamental form of $\hat F, F$ respectively. Then  $\hat F, F$ are congruent. 
\end{proposition}
\begin{remark}
 The global rigidity holds for round spheres via uniqueness of CMC spheres, while the infinitesimal rigidity fails for round spheres  as demonstrated in Theorem~\ref{th:infinitesimal}.

\end{remark}
\begin{proof}
We denote the geometric quantities associated with $\hat g$ and $g$ by an hat and without an hat respectively. By taking the trace of the second fundamental form, those two immersions give the same induced mean curvature $\hat H = H$. 

We may assume that $\hat F(\Sigma)$ is not umbilic; otherwise $\hat H$ and $H$ are both constant, and the desired conclusion follows directly from Hopf's uniqueness theorem of immersed spheres of constant mean curvature. 

The Gauss equation implies the scalar curvatures $\hat R_\Sigma=R_\Sigma$. Together with the  conformal transformation formula  $\hat R_\Sigma = u^{-1} (R_\Sigma -\Delta_\Sigma u + u^{-2} |\nabla_\Sigma u|^2)$,   we obtain 
\begin{align}\label{eq:con}
	\Delta_\Sigma u- u^{-2} |\nabla_\Sigma u|^2 = R_\Sigma (1-u).
\end{align}
By the Codazzi equation, $\Div_{\hat \gamma} \hat  A = dH = \Div_\gamma A$, and we compute 
\begin{align}
	0 = \Div_\gamma (u^{-1} \hat  A - A) &= \Div_{\hat  \gamma} \hat  A + \tfrac{1}{2} u^{-1} d u H - u^{-2} u_{,k} \hat A_{i\ell} \gamma^{k\ell} - \Div_\gamma A \notag\\
	&= \tfrac{1}{2} u^{-1} d u H - u^{-2} u_{,k} \hat A_{i\ell} \gamma^{k\ell} \notag\\
	&= - u^{-2} \big(\hat  A- \tfrac{1}{2} H \hat  \gamma\big)(\nabla_\gamma u, \cdot).\label{eq:con2}
\end{align}

Define the open subset   $\Sigma_0\subset \hat F(\Sigma)$ consisting of the points where $\hat R_\Sigma \neq 0$. Note that $\Sigma_0\neq \emptyset$. We claim that the traceless part of the second fundamental form  $\hat  A- \tfrac{1}{2} H \hat  \gamma$ is not  identically zero on $\Sigma_0$.  We may assume that $\Sigma_0$ is not the entire $\hat F(\Sigma)$; otherwise the claim follows. 

 Suppose, to get a contradiction, $\hat  A- \tfrac{1}{2} H \hat  \gamma\equiv 0$ on $\Sigma_0$. Then by the Codazzi equation $\hat H$ is locally constant on $\Sigma_0$, and then so is $\hat R_\Sigma$ by the Gauss equation. It implies that $\hat R_\Sigma$ is identically zero. A contradiction. 

Let $U$ be the subset of $\hat F(\Sigma)$ so both $\hat R_\Sigma \neq 0$ and $\hat  A- \tfrac{1}{2} H \hat  \gamma \neq 0$. Then by \eqref{eq:con} and \eqref{eq:con2}, $u\equiv 1$ on $U$, and hence $u\equiv 1$ on $\Sigma$ by unique continuation. By Fundamental Theorem of Surfaces $\hat  F, F$ are congruent.
\end{proof}

\section{Desingularization} \label{se:remove}

In this section, we discuss the procedure to ``desingularize" the coordinate singularity in the Euclidean Schwarzschild, anti-de Sitter Schwarzschild metric, and de Sitter Schwarzschild by periodizing the ``time'' coordinate $t$ to a circle of appropriate period. While this method is familiar to experts, we had difficulties in locating literature covering the cases of $\Lambda \neq 0$ and the higher dimensional cases $n>4$. Below, we present a unified approach and determine the explicit period for $t$ in all cases. 
 
 For each $n\ge 4$,  $\Lambda \in \mathbb R$, and $m>0$, we define the function $f(r)=1-2mr^{3-n}-\Lambda r^2$. When $\Lambda>0$, we furthermore assume 
  \begin{align}
 m^2\Lambda^{n-3}<\frac{(n-3)^{n-3}}{(n-1)^{n-1}}.\tag{\ref{eq:Lambda}}
 \end{align} 
 Then, as discussed in Section~\ref{se:example}, $f(r)$ has a unique positive root $r_0$ when $\Lambda \le 0$, while $f(r)$ has two positive roots $r_0 < r_1$ when $\Lambda > 0$. In the following, we will simultaneously address all cases of $\Lambda$, making it convenient to denote $r_1 = +\infty$ for the case $\Lambda \le 0$.

Our goal is to show that the metric
\[
g= f(r)dt^2+\frac{1}{f(r)}dr^2+r^2 g_{S^{n-2}},
\]
can be extended at $r=r_0$. It is sufficient to extend the two dimensional metric $f(r)dt^2+\frac{1}{f(r)}dr^2$. Consider the following change of coordinates. Let $F:[r_0, r_1)\to \mathbb R$ be a smooth function satisfying $F(r_0)=0$ and $F'(r)>0$ and let $\ell$ be a positive constant, both to be determined below. Define the new coordinates
\[
	x = \sqrt{F(r)} \cos \left(\frac{t}{2\ell}\right) \quad \mbox{ and } \quad y= \sqrt{F(r)} \sin\left (\frac{t}{2\ell}\right)
\]  
where $\frac{t}{2\ell} \in [0, 2\pi)$. It is direct to compute 
\begin{align*}
\frac{4\ell }{F'(r)} (dx^2 + dy^2) &= \frac{\ell F'(r)}{F(r)} dr^2 + \frac{F(r)}{\ell F'(r)} dt^2.
\end{align*}

We show in the proposition below that there is  a positive constant $\ell>0$ and a smooth function $F: [r_0, r_1)\to \mathbb R$  so that 
\begin{align}\label{eq:F}
F(r_0)=0, \quad F'(r)>0,\quad \mbox{ and } \quad \frac{\ell  F'(r)}{F(r)} = \frac{1}{ f(r)}, \quad \mbox{ for all } r\ge r_0.
\end{align}
It then implies that 
\[
f(r) dt^2 + \frac{1}{f(r)} dr^2 = \frac{4\ell }{F'(r)} (dx^2 + dy^2)
\]
and thus the metric can be extended at $r=r_0$.

\begin{proposition}
Let $\ell =  \frac{1}{(n-3) r_0^{-1} - (n-1)\Lambda r_0} >0$. Then there is a smooth function $F:[r_0, r_1)\to \mathbb R$ satisfying \eqref{eq:F}.
\end{proposition}
\begin{proof}
It is straightforward to see that $r_0$ is a root of the polynomial $r^{n-3} f(r)$ with multiplicity $1$, so $r^{n-3} f(r)=(r - r_0) p(r)$, where $p(r)$ is a polynomial  and $p(r) > 0$ for all $ r \in [r_0,  r_1)$. By the Taylor expansion of $\frac{r^{n-3}}{p(r)}$ at $r=r_0$, we have 
\begin{align*}
	\frac{1}{f(r)} = \frac{r^{n-3}}{(r-r_0) p(r)} = \frac{r_0^{n-3}}{p(r_0)(r-r_0)} + b_0(r)
\end{align*}
where $b_0(r)$ is a smooth function defined for all $r\in [r_0, r_1)$. Denote the positive constant 
\[
a=  \frac{r_0^{n-3}}{\ell p(r_0)}
\] 
and note
\[
	p(r_0 )=\lim_{r\to r_0} \frac{r^{n-3} f(r)}{r-r_0} = (n-3)r_0^{n-4} - (n-1)\Lambda r_0^{n-2} > 0. 
\]
Therefore,  
\begin{align*}
\int_{r_0}^r \frac{1}{\ell f(s)} \,ds &= \int_{r_0}^r \frac{a}{ (s-r_0)} \, ds + \int_{r_0}^r \frac{1}{\ell} b_0(s)\, ds=\log (s-r_0)^a \Big|_{s=r_0}^{s=r} + b_1(r)
\end{align*}
where $b_1(r)$ is a smooth function for all $r\in [ r_0, r_1)$.  

We define $F(r)$ as follows: 
\begin{align*}
	F(r) &= e^{\int_{r_0}^r \frac{1}{\ell f(s)} \,ds}=(r-r_0)^a e^{b_1(r)}
\end{align*}
and compute 
\begin{align*}
	F'(r) &=\frac{F(r)}{\ell f(r)}=\frac{(r-r_0)^a r^{n-3}e^{b_1(r)}}{\ell \, (r-r_0)p(r)}=\frac{(r-r_0)^{a-1} r^{n-3}e^{b_1(r)}}{\ell \, p(r)}.
\end{align*}
If we require the exponent $a=1$, or equivalently 
\[
 \ell =  \frac{1}{(n-3) r_0^{-1} - (n-1)\Lambda r_0},
\]
it is direct to see that $F(r)$ defined above satisfies all the properties of \eqref{eq:F}.
\end{proof}
\begin{remark}
For $\Lambda>0$, one can also apply the same procedure at the other root $r=r_1$ and show that by letting 
\[
 \hat \ell =  -\frac{1}{(n-3) r_1^{-1} - (n-1)\Lambda r_1},
\]
one can extend $f(r) dt^2 + \frac{1}{f(r)} dr^2$ to be defined at $r=r_1$ for  $\frac{t}{2\hat \ell} \in [0, 2\pi)$. Note that the function $(n-3) r^{-1} - (n-1)\Lambda r$ is strictly decreasing for $r>0$, so $\ell$ is never equal to $\hat {\ell}$. Therefore, the desingularized metric must have a conical singularity at the other root.

\end{remark}

\bibliographystyle{amsplain}
\bibliography{2023}
\end{document}